\documentclass{amsart}

\usepackage{amsmath,amssymb,amsthm,amsfonts}
\usepackage{hyperref}

\usepackage{color}

\newtheorem{lemma}{Lemma}[section]
\newtheorem{theorem}{Theorem}[section]

\newtheorem{remark}{Remark}[section]
\newtheorem{corollary}{Corollary}[section]

\numberwithin{equation}{section}

\newcommand{\dis}{\displaystyle}

\newcommand{\R}{\mathbb{R}}

\newcommand{\Z}{\mathbb{Z}}

\newcommand{\CD}{\mathcal{D}}
\newcommand{\CE}{\mathcal{E}}
\newcommand{\CF}{\mathcal{F}}

\newcommand{\CM}{\mathcal{M}}

\newcommand{\na}{\nabla}

\newcommand{\al}{\alpha}
\newcommand{\be}{\beta}
\newcommand{\ga}{\gamma}

\newcommand{\la}{\lambda}
\newcommand{\de}{\delta}
\newcommand{\si}{\sigma}
\newcommand{\pa}{\partial}

\newcommand{\eps}{\epsilon}

\newcommand{\De}{\Delta}

\newcommand{\lag}{\langle}
\newcommand{\rag}{\rangle}

\begin{document}

\title[{OPTIMAL DECAY RATES TO SOME CONSERVATION LAWS}]
{Optimal Decay Rates to Conservation Laws with Diffusion-Type Terms
of Regularity-gain and Regularity-loss}

\author{Renjun Duan}
\address{(RJD)
Department of Mathematics, The Chinese University of Hong Kong,
Shatin, Hong Kong} \email{rjduan@math.cuhk.edu.hk}

\author{Lizhi Ruan}
\address{(LZR)
The Hubei Key Laboratory of Mathematical Physics, School of
Mathematics and Statistics, Central China Normal University, Wuhan
430079, P.R. China} \email{rlz@mail.ccnu.edu.cn}

\author{Changjiang Zhu}
\address{(CJZ)
The Hubei Key Laboratory of Mathematical Physics, School of
Mathematics and Statistics, Central China Normal University, Wuhan
430079, P.R. China} \email{cjzhu@mail.ccnu.edu.cn}
\thanks{Corresponding author. Email: cjzhu@mail.ccnu.edu.cn (C.J.
Zhu)}

\date{\today}
\maketitle

\thispagestyle{empty}

\begin{abstract}
We consider the Cauchy problem on nonlinear scalar conservation laws with a
diffusion-type source term related to an index $s\in \R$ over the
whole space $\R^n$ for any spatial dimension $n\geq 1$. Here,  the
diffusion-type source term behaves as the usual diffusion term over the low frequency domain while it admits on the high frequency part a feature of regularity-gain and
regularity-loss for $s< 1$ and $s>1$, respectively. For all
$s\in \R$, we not only obtain the $L^p$-$L^q$ time-decay estimates on the linear solution
semigroup  but also establish the global existence and optimal
time-decay rates of small-amplitude classical solutions to the
nonlinear Cauchy problem. In the case of
regularity-loss, the time-weighted energy method is introduced to
overcome the weakly dissipative property of the equation. Moreover,
the large-time behavior of solutions asymptotically tending to the
heat diffusion waves is also studied. The current results have
general applications to several concrete models arising from
physics.
\end{abstract}



\tableofcontents


\section{Introduction}

\subsection{Main results}
In this paper, we consider the Cauchy problem on a scalar
conservation law with a diffusion-type source term, taking the form of
\begin{equation}\label{CP.eq}
\left\{\begin{array}{l}
 \dis  \pa_t u+\na\cdot f(u)=\De P_su,\ \  x\in \R^n, t>0,\\[3mm]
\dis u|_{t=0}=u_0,\ \  x\in \R^n.
\end{array}\right.
\end{equation}
Here, $u=u(x,t): \R^n\times (0,\infty)\to \R$ is unknown and
$u_0=u_0(x): \R^n\to \R$ is given. $n\geq 1$ denotes the spatial
dimension. $f(u)=(f_1(u),f_2(u),\cdots,f_n(u)): \R\to \R^n$ is a
given smooth flux function. Since only the small amplitude classical solutions will be discussed,  we suppose without loss of
generality that
\begin{equation}\label{s4.assu.f}
    f_j(0)=f_j'(0)=0, \ \ 1\leq j\leq n.
\end{equation}
Otherwise, one can take change of variables
\begin{equation*}
    \tilde{t}=t,\ \ \tilde{x}_j=x_j+f_j'(0)t,
\end{equation*}
and denote $\tilde{f}(\cdot)$ by
\begin{equation*}
    \tilde{f}(u)=f(u)-f(0)-f'(0)u,
\end{equation*}
so that the form of \eqref{CP.eq} remains unchanged but
\eqref{s4.assu.f} still holds for $\tilde{f}$. In the diffusion-type
source term $\De P_s u$, $P_s$ with an index $s\in \R$ is a
pseudo-differential operator defined by
\begin{equation*}
P_s u=\CF^{-1}\{[m(\xi)]^{-s}\CF u\},
\end{equation*}
where through this paper the frequency function $m(\xi)$ is supposed to be strictly positive and continuous in
$\xi\in \R^n$. Moreover, the following assumptions on $m(\xi)$ which
describe its more precise behavior as $|\xi|$ is close to infinity or zero will also be used:
\begin{itemize}
\item[$(\mathcal{M}_1)$] Whenever $|\xi|\gg 1$ is close to infinity, $m(\xi)\sim |\xi|^2$. Notice that since $m(\xi)$
is strictly positive and continuous in $\xi\in \R^n$, it is equivalent with the assumption that
$m(\xi)\sim 1+|\xi|^2$ for $\xi\in \R^n$, that is, there are constants $m_1\geq m_0>0$ such that
for any $\xi\in \R^n$,
\begin{equation}\label{m.form}
    m_0(1+|\xi|^2)\leq m(\xi)\leq m_1(1+|\xi|^2).
\end{equation}
\item[$(\mathcal{M}_2)$] Whenever $|\xi|\ll 1$ is close to zero,
$m(\xi)-m(0)$ vanishes with rate $|\xi|^{\si}$ for a constant
$\si>0$. Precisely, there is $m_2>0$ such that
\begin{equation*}
    |m(\xi)-m(0)|\leq m_2|\xi|^\si
\end{equation*}
whenever $|\xi|$ is small enough.
\end{itemize}
\noindent As will be seen later on, several mathematical models in
different contexts can provide some concrete examples related to
equation \eqref{CP.eq}$_1$ for different forms of $m(\xi)$
satisfying assumptions $(\CM_1)$ and $(\CM_2)$.

The aim here is to study the well-posedness and large-time behavior
on the Cauchy problem \eqref{CP.eq} for arbitrary space dimensions
$n\geq 1$ and for all $s\in \R$ in the framework of small-amplitude
classical solutions. Note that under the assumption $(\CM_1)$,
\begin{equation}\label{De.sim}
    \De P_s u=-\CF^{-1}\left\{\frac{|\xi|^2}{m(\xi)^s}\CF u\right\}\ \
    \text{with}\ \
    \frac{|\xi|^2}{m(\xi)^s}\sim \frac{|\xi|^2}{(1+|\xi|^2)^s},
\end{equation}
and also the regularity of solutions is generally determined by their high frequency part.
Thus, it is convenient to say that equation \eqref{CP.eq}$_1$ is of
the regularity-gain type for {$s< 1$} whereas of the
regularity-loss type for $s>1$. In the critical case when $s=1$, $\De P_s u$ just behaves like a damping term over the high frequency domain;
more discussions will be given in the next subsection.  In this paper, for all $s\in \R$,  we shall establish under $(\CM_1)$
the global existence of solutions and the time-decay rate of
solutions and their derivatives up to some order,
where the extra regularity on initial data is required in the case
of the regularity-loss type. Moreover, we shall also prove that the
obtained solution for all $s\in \R$ time-asymptotically tends to
some heat diffusion wave if $(\CM_2)$ is further assumed. Precisely, these results are given as follows.

First of all, we are  concerned with the global existence and
time-decay rate of solutions to the Cauchy problem \eqref{CP.eq}
under the assumption $(\CM_1)$. When $s\leq 1$,  one has

\begin{theorem}[case of $s\leq 1$]\label{thm.s.small}
Let $n\geq 1$, $s\leq 1$, and let $N\geq [n/2]+2$. Suppose the
assumption $(\CM_1)$ holds. There is a constant $\eps_0>0$ such that if
$\|u_0\|_{H^N}\leq \eps_0$, then the Cauchy problem \eqref{CP.eq}
admits a unique global solution $u(x,t)$ satisfying
\begin{equation*}
 u\in C^0\left([0,\infty); H^N\right)\cap C^1\left([0,\infty);
H^{N-1}\right)
\end{equation*}
and
\begin{equation*}
\left\|u(t)\right\|_{H^N}^2+\int_0^t\left\|\nabla \lag \na
\rag^{N-s}u(\tau)\right\|^2d\tau \leq C\left\|u_0\right\|_{H^N}^2
\end{equation*}
for any $t\geq 0$. Furthermore, whenever
$\|u_0\|_{H^N\cap L^1}$ is small enough, the obtained solution $u$
enjoys the time-decay estimate
\begin{equation}\label{thm.s.small.decay}
\left\|\nabla^ku(t)\right\|\leq C\|u_0\|_{H^N\cap
L^1}(1+t)^{-\frac{n}{4}-\frac{k}{2}}
\end{equation}
for any $t\geq 0$, where $0\leq k\leq N$.
\end{theorem}

\begin{remark}
As far as the time-decay estimate \eqref{thm.s.small.decay} is
concerned, {the optimal $L^2$-decay estimate is obtained for all derivatives  up
to $N$ order.} The smallness assumption on
$\|u_0\|_{L^1}$ is used only for the case of derivatives $\na^k u$
with $1\leq k\leq N$, but it is not necessary when $k=0$; see Remark
\ref{rem.small.L1} for details. For $s=1$, the corresponding result
in Theorem \ref{thm.s.small} has been obtained in \cite{DFZ}, where
the time-decay of the $L^2$-norm for $u(x,t)$ only was considered.
\end{remark}

Although the existence result stated as in Theorem \ref{thm.s.small}
is more or less standard for the case when $s\leq 0$ for which
equation \eqref{CP.eq}$_1$ gains at least one-order derivative
regularity, we would use a unified elementary energy method as in
\cite{DFZ} to prove the whole case of $s\leq 1$, which in turn can
shed light on some main difficulties for the regularity-loss case of
$s>1$. On the other hand, it is non trivial to obtain the optimal
time-decay rate of the highest-order derivative $\na_x^Nu(t)$ in
$L^2$-norm without putting additional regularity on the initial data
$u_0$. From the proof later on, this is actually based on some
interpolation estimate on the $\na^{N}\lag \na \rag^{-s}$-type
derivative of $u(x,t)$; see Lemma \ref{lem.inter}.

Next, when $s>1$, the result similar to that in Theorem \ref{thm.s.small} is stated as follows.

\begin{theorem}[case of $s> 1$]\label{thm.s.large}
Let $n\geq 1$, $s> 1$, and let $N\geq N_0$, where $N_0$ is defined
by \eqref{def.N0} in terms of $n$ and $s$. Suppose the assumption
$(\CM_1)$ holds. There is a constant $\eps_1>0$ such that if
$\|u_0\|_{H^N}+\|u_0\|_{L^1}\leq \eps_1$, then the Cauchy problem
\eqref{CP.eq} admits a unique global solution $u(x,t)$ satisfying
\begin{equation}\label{thm.s.large.1}
 u\in C^0\left([0,\infty); H^N\right)\cap C^1\left([0,\infty);
H^{N-1}\right).
\end{equation}
Furthermore, the obtained solution also enjoys the time-decay estimate
\begin{equation}\label{thm.s.large.2}
\left\|\nabla^ku(t)\right\|\leq
C\|u_0\|_{H^N\cap L^1}(1+t)^{-\frac{n}{4}-\frac{k}{2}}
\end{equation}
for any $t\geq 0$, where $0\leq k\leq N_1$ with $N_1$ defined by \eqref{def.N1} in terms of $n$, $s$ and $N$.
\end{theorem}

For this time, let us point out some obvious differences of the
above two theorems corresponding to the case when $s\leq 1$ and
$s>1$. First, for the global existence of solutions, the smallness
condition on $\|u_0\|_{L^1}$ is needed in the case when $s>1$ but
not required when $s\leq 1$. Second, from the definition \eqref{def.N0}, $N_0$
standing for the smoothness degree of initial data $u_0$ is strictly
larger than $\left[\frac{n}{2}\right]+2$, which shows that the
additional regularity of initial data is necessary for the case of
$s>1$ so as to guarantee the global existence of solutions to the
Cauchy problem \eqref{CP.eq}. Last, from the definition
\eqref{def.N1} of $N_1$,
\begin{equation*}
  \left[\frac{n}{2}\right]+2\leq N_1 <N.
\end{equation*}
This together with \eqref{thm.s.large.2} imply that under the same
condition on $H^N$ initial data, the optimal time-decay rate in the
case of $s>1$ is obtained for $u$ and its lower-order derivatives
only, but not for all derivatives up to $N$-order. Here, the optimal
rate means that it is equal to the one in the case of the linearized
equation; see Lemma \ref{lem.decay}. We remark that all these
differences between Theorem \ref{thm.s.small} and Theorem \ref{thm.s.large} result from the special feature of the equation for
different values of $s$, that is, equation \eqref{CP.eq}$_1$ is of
the regularity-gain type for $s<1$ but of the regularity-loss type
for $s>1$.

Finally, we consider the asymptotic behavior of the obtained-above solutions for all
$s\in \R$. It turns out that for either $s\leq 1$ or
$s>1$,  a properly-defined heat diffusion wave can be regarded as a
good time-asymptotic profile. For that, set $\mu_s=1/m(0)^s$. Note
$\mu_s>0$. Define the Green function $G=G(x,t)$ by
\begin{equation*}
G(x,t)=(4\mu_s\pi
t)^{-\frac{n}{2}}e^{-\frac{|x|^2}{4\mu_st}}=\CF^{-1}\{e^{-\mu_s|\xi|^2}\}.
\end{equation*}
The result is stated as follows.

\begin{theorem}[time-asymptotic behavior]\label{thm.asymptotic}
Let $n\geq 2$, $k\geq 0$  be integers and $s\in \R$. Suppose that
both assumptions $(\CM_1)$ and $(\CM_2)$ on $m(\xi)$ hold.  Let
$u=u(x,t)$ be the solution to the Cauchy problem \eqref{CP.eq}
constructed in Theorems \ref{thm.s.small} or \ref{thm.s.large} for
given initial data $u_0\in H^N\left({\mathbb{R}^n}\right)\cap
L^1\left({\mathbb{R}^n}\right)$, where $N$ is chosen such that if
$s\leq 1$ then $N\geq \left[\frac{n}{2}\right]+2$ and if $s>1$ then
$N\geq N_0$. Set $u^\ast=u^\ast(x,t)$ by
\begin{equation}\label{def.heatwave}
 u^\ast(x,t)=G(x,t+1)\int_{\mathbb{R}^n}u_0(x)dx.
\end{equation}
Then, whenever $\|u_0\|_{H^N\cap L^1}$ is further  small enough and
$u_0\in L^1_1$ with $\int_{\R^n}u_0\,dx=0$,
\begin{equation}\label{thm,asy.large.dec}
\left\|\nabla^k\left(u-u^*\right)(t)\right\|\leq C\|u_0\|_{H^N\cap
L_1^1}\rho(t)(1+t)^{-\frac{n}{4}-\frac{k}{2}-\frac{1}{2}\min\{1,\si\}}
\end{equation}
for any $t\geq 0$, where $\rho(t)=\ln(1+t)$ for $n=2$ and
$\rho(t)=1$ for $n\geq 3$, and $k$ is chosen in terms of $n$, $s$
and $N$ such that if $s\leq 1$ then $0\leq k\leq N$ and if $s>1$
then $0\leq k\leq N_2$ with $N_2$ defined in \eqref{def.N2}.
\end{theorem}

\begin{remark}
From \eqref{def.heatwave}, $u^\ast(x,t)$ carries  the same mass of initial data and it is usually called the heat
diffusion wave.  $u^\ast(x,t)$ is a good time-asymptotic profile of solutions to
the Cauchy problem \eqref{CP.eq} in the sense that the rate of the
nonlinear solution $u(x,t)$ converging to the profile $u^\ast(x,t)$
is strictly larger than one of the nonlinear solution $u(x,t)$
itself decaying to zero due to $\si>0$ and hence $\min\{1,\si\}>0$
by the assumption $(\CM_2)$.
\end{remark}

\subsection{Applications and related literature}
In this subsection we review some known results related to the
general model
\begin{equation}\label{eq.app}
\pa_t u+\na \cdot f(u)=\De P_s u
\end{equation}
proposed in this paper. As will be discussed in the following, this
general model can be connected to some concrete equation or system
by taking different values of $s$, which in fact have been
extensively studied even for a lot of different issues in individual
contents. We remark that the exhaustive literature list is beyond
the scope of the paper, and thus only some closely related results
will be mentioned; interested readers can refer to them and
references therein.

When $s$ takes the negative value, similar to \eqref{De.sim},
\begin{equation*}
\De P_s\sim \De-(-\De)^{1-s}
\end{equation*}
in the sense of the frequency space.  Particularly when $s=-1$, it
is connected with the Cahn-Hilliard (CH) equation describing the
phase separation in binary alloys, in the form of
\begin{equation}\label{CH}
\dis  \pa_tu+\De \varphi(u)+\De^2u=0,
\end{equation}
where $\varphi(\cdot)$ is a given function. Notice that if $\varphi$
is smooth and $\varphi'(0)<0$, the linearized equation of \eqref{CH}
is written as
\begin{equation*}
\pa_t u=-\varphi'(0)\De u-\De^2 u
\end{equation*}
which exactly corresponds to the case of $s=-1$ in the linearized
level. For the related study of {the CH equation, we mention
\cite{CaMu} and \cite{Liu-W-Zhao} only.}

{For the Cauchy problem in general $n$ dimensions, \cite{CaMu}
studied the following problem:
\begin{equation}\label{CaMu}
\left\{\begin{array}{l}
 \dis  \partial_tu=\De u+\sum\limits_{ij=1}^n
 \frac{\partial^2f_{ij}(u)}{\partial_{x_i}\partial_{x_j}u}-\varepsilon^2\De^2u,\ \  x\in \R^n, t>0,\\[5mm]
\dis u|_{t=0}=u_0,\ \  x\in \R^n.
\end{array}\right.
\end{equation}
Under the assumptions that $f_{ij}(u)\ (i,j=1,2,\cdots,n)$ are
Lipschitz continuous and equal to a constant outside a bounded
interval in $u$, $u_0(x)\in L^\infty(\mathbb{R}^n)$,
$\varepsilon \nabla u_0(x)\in L^\infty(\mathbb{R}^n)$ and under some additional conditions, an $L^\infty(\mathbb{R}^n)$-a priori estimate,
which is independent of $\varepsilon>0$, is obtained for the
solution of the Cauchy problem \eqref{CaMu}. On the other hand, few results on the
temporal decay estimates and on the asymptotics of the solution of
the corresponding Cauchy problem have been obtained. Thus two natural
questions are under which conditions on the smooth nonlinear functions
$f_{ij} (u) \ (i, j = 1, 2,\cdots, n)$ the corresponding Cauchy
problem admits a unique global smooth solution $u(x,t)$, and how to
get the optimal temporal decay estimates and further describe its
time-asymptotic profile.}

For the above questions, an answer was given in  \cite{Liu-W-Zhao}  for a special case of the CH equation. Precisely,  for {\eqref{CH}},  when $\varphi'(0)=0$, \cite{Liu-W-Zhao} proved
the global existence of the $C^\infty$ solution $u$ for initial data
$u_0\in L^1\cap L^\infty$ with $\|u_0\|_{L^1}$ small enough, and
further obtained the optimal time-decay estimates
\begin{equation*}
\left\|\na^k u(t)\right\|_{L^p\left(\mathbb{R}^n\right)}\leq
C(\tau)(1+t)^{-\frac{k}{4}-\frac{n}{4}\left(1-\frac{1}{p}\right)},\
\ t\geq\tau>0,\ \ k=0,1,\cdots,
\end{equation*}
where $1\leq p\leq \infty$. Moreover, for the asymptotic profile,
under some additional condition on $\varphi$, it was also proved in
\cite{Liu-W-Zhao} that
\begin{equation*}
\lim\limits_{t\rightarrow\infty}(1+t)^{\frac{n}{4}\left(1-\frac{1}{p}\right)}
\left\|u(t)-\CF^{-1}(e^{-|\xi|^4t})\int_{\R^n}u_0\,
dx\right\|_{L^p\left(\mathbb{R}^n\right)}=0.
\end{equation*}
As pointed out in \cite[Remark 1.2]{Liu-W-Zhao}, the case of
$\varphi'(0)<0$ can be considered in a similar way, which is related
to our results for small-amplitude $H^N$ initial data stated before.

When $s=0$, one can simply let $P_s$ be an identity operator and
thus  \eqref{eq.app} is equivalent with the usual scalar viscous
conservation law
\begin{equation*}
\dis  \pa_tu+\na\cdot f(u)=\De u.
\end{equation*}
The study of the above equation has a long history and has been much
more extensively investigated from different respects, particularly
well-posedness and finite-time blow up of solutions to the Cauchy
problem even for large initial data or general flux function, and
also stability of wave patterns such as the smooth shock wave,
rarefaction wave and contact discontinuity, {cf. \cite{Daf}.}

When $s=1$, a typical model related to \eqref{eq.app} takes the form
of
\begin{equation}\label{RCE}
\dis  \pa_t u+\na\cdot f(u)=\De P_1^{\eps}u,\ \ \De
P_1^{\eps}u:=\CF^{-1}\left\{\frac{-\varepsilon
|\xi|^2}{1+\varepsilon^2|\xi|^2}\CF u\right\},
\end{equation}
where $\eps>0$ is a parameter. This model was derived in
\cite{Ro-PRA} as the corresponding extension of the Navier-Stokes
equations via the regularization of the Chapman-Enskog expansion
from the Boltzmann equation, which is intended to obtain a bounded
approximation of the linearized collision operator for both low and
high frequencies. As pointed out in \cite{Ro-PRA, ScTa}, we remark
that the right-hand term of \eqref{RCE}  can be understood in the
sense that $\De P_1^\eps u$ behaves qualitatively like the usual
Navier-Stokes viscosity $\De u$ at the low frequency, while it
essentially acts as the damping force $-u$ at the high frequency.
The first rigorous mathematical study of \eqref{RCE} is given in
\cite{ScTa} for the one space dimension $n=1$, where the propagation
of smoothness of solutions with small initial data, existence of
traveling waves, existence of entropy solutions with BV initial data
and zero relaxation limit as $\eps\to 0$ were considered.

On the other hand, \eqref{RCE} with $\eps=1$ can be recovered as in
\cite{DFZ} from the following hyperbolic-elliptic coupled system
\begin{equation}\label{HE}
\left\{
\begin{array}{l}
\pa_tu+ \na\cdot [f(u) + q]= 0,\\[3mm]
- \na \na\cdot q+q+\na u=0.
\end{array}
\right.
\end{equation}
When $n=1$ and $f(u)=\frac{1}{2}u^2$, \eqref{HE} further reduces to
a simplified model system
\begin{equation}\label{Hamel}
    \left\{\begin{array}{l}
      \pa_t u+\pa_x (\frac{1}{2}u^2+q)=0,\\[3mm]
      -\pa_x^2 q+q+\pa_x u=0.
    \end{array}\right.
\end{equation}
The above system, which was first derived in \cite{Ham}, actually
arises from the study of radiation hydrodynamics. A more general
system  describing the one-dimensional motion of the radiating fluid
takes the form of
\begin{equation}\label{s1.M2}
    \left\{\begin{array}{l}
      \pa_t \rho +\pa_x (\rho u)=0,\\[3mm]
      \pa_t (\rho u)+\pa_x (\rho u^2 +p)=0,\\[3mm]
      \pa_t [\rho (e+\frac{1}{2}u^2)] +\pa_x[\rho u  (e+\frac{1}{2}u^2)+p
      u+q]=0,\\[3mm]
      -\pa_x^2 q+3 a^2 q +4 a \kappa \pa_x \theta^4=0,
    \end{array}\right.
\end{equation}
where the hydrodynamic functions $\rho\geq 0$, $u$, $p$, $e$ and
$\theta\geq 0$ denotes the mass density, velocity, pressure,
internal energy and absolute temperature of the fluid, respectively,
and $q$ is the radiative heat flux, and $a>0$, $\kappa>0$ are the
absorption coefficient and the Boltzmann constant, respectively; see
\cite{ViKr}.

For some mathematical results on models {\eqref{HE} and
\eqref{Hamel},} interested readers can refer to \cite{DFZ}. Here, we
only mention the study of \eqref{HE} in the case of high space
dimensions. In fact, the model \eqref{HE} over $\R^n$ for any $n\geq
1$ was first proposed in \cite{Fra-NDEA} by taking the proper
approximation of the high dimensional version of \eqref{s1.M2},
where the one-dimensional result in \cite{LaMa} was generalized to
the case of several dimensions and also the existence and uniqueness
of global weak, entropy solutions to the initial value problem and
the two different relaxation limits were studied. Recently, the
stability and convergence rate of solutions near constant states or
planar rarefaction waves were obtained in
\cite{GaZh1-M3AS,GaZh2-JDE,Ruan-Zhu} for the case  when the spatial
dimension takes values $2\leq n\leq 8$ on the basis of $L^p$-energy
method. For any $n\geq 1$, \cite{WaWa} exposed the pointwise
estimate of solutions by using the Green's function method, and
\cite{DFZ} also developed a refined energy method to consider the
stability of constants states, smooth planar waves and time-periodic
solutions in the presence of the time-periodic source. Very
recently, by using a time-weighted energy method,
\cite{Liu-Kawashima} obtained the global existence and optimal decay
estimates of classical solutions with small amplitude, and also
showed that the solution tends time-asymptotically to the linear
heat diffusion wave.

Next, for the general model \eqref{eq.app}, we turn to the case
when $s$ is strictly larger than one. If one takes $s=2$ and
$m(\xi)=(1+|\xi|^2+|\xi|^4)^{1/2}$, similarly as in \cite{DFZ}, the
corresponding model \eqref{eq.app} can be recovered from the
hyperbolic-elliptic system with a fourth-order elliptic part
\begin{equation*}
\left\{
\begin{array}{l}
\pa_t u+ \na\cdot [f(u) + q]= 0,\\[3mm]
\De^2q- \na \na\cdot q+q+\na u=0.
\end{array}
\right.
\end{equation*}
The above system  was considered in \cite{HoKa} for one space
dimension $n=1$, where authors established the global solvability
and asymptotic behavior of solutions for small initial data in
$H^N\cap L^1$ with $N\geq 7$, and also obtained the optimal
time-decay rates
\begin{equation*}
\left\|\partial_x^ku(t)\right\|\leq
C(1+t)^{-\frac{1}{4}-\frac{k}{2}}
\end{equation*}
for $0\leq k\leq [(N+1)/2]-2$. Later on, \cite{Kuka} considered the
one dimensional version of the following more general model
\begin{equation}\label{AHE}
\left\{
\begin{array}{l}
\pa_t u+ \na\cdot [f(u) + q]= 0,\\[3mm]
(-\De)^sq+q+\na u=0,
\end{array}
\right.
\end{equation}
for $s\geq 2$, which is indeed equivalent with \eqref{eq.app} with
$m(\xi)=(1+|\xi|^{2s})^{1/s}\sim 1+|\xi|^2$. For the model
\eqref{AHE} when $n=1$ and $s\geq 2$, \cite{Kuka} proved some
similar results as in  \cite{HoKa}. It should be pointed out that
both \cite{HoKa} and \cite{Kuka} also showed that the solution
approaches the nonlinear heat diffusion wave described by the
self-similar solution of the viscous Burgers equation as time tends
to infinity, and furthermore they proposed the method to deal with
the time-decay property of solutions for those equations of
regularity-loss type.

For the regularity-loss phenomenon, it actually has been observed for some other realistic systems of equations. For instance, \cite{DS-VMB} considered the optimal large-time behavior of solutions to the Vlasov-Maxwell-Boltzmann system
that  describes the dynamics of the kinetic plasma, and a similar but more elaborate result was also obtained in \cite{Duan-EM} for the Euler-Maxwell system in the context of fluid plasma. In addition, \cite{IHK} and \cite{IK} studied the time-decay property for the dissipative Timoshenko system. All these mentioned systems admit a common structure that they are of the regularity-loss type. A typical feature that this kind of regularity-loss structure generates is that the linear solution semigroup has an upper-bound estimate by $e^{-p(\xi)t}$, where the frequency function $p(\xi)$ behaves like $|\xi|^2/(1+|\xi|^2)^2$ corresponding to the case of $s=2$ that we discuss here. We believe that the current developed approach could become a powerful tool to provide much more elaborate results than those in \cite{DS-VMB,Duan-EM}.

Finally, we emphasize that although the general model that we propose here can cover several concrete examples in different physical contexts mentioned above, we work  only in the framework of the small amplitude classical solution, which thus makes it possible to found a general theory of the global existence and large-time behavior of solutions. On the other hand, it could be very interesting to consider some other issues in the current setting of the model, such as the well-posedness of the Cauchy problem with large initial data and finite time below up of solutions; these are left to our future study.

\subsection{Notations and arrangement of the paper}

Through this paper, $C$ denotes a generic positive (generally large)
constant and $\la$ denotes a generic positive (generally small)
constant.  For an integer $m\geq 0$, we use $H^m$ denotes the
Sobolev space $H^m(\R^n)$ with norm $\|\cdot\|_{H^m}$, and set
$L^2=H^0$ with norm $\|\cdot\|$ when $m=0$.
$\langle\cdot,\cdot\rangle$ denotes the inner product in $L^2$.
$L_1^1$ denotes the $(1+|x|)$-weighted $L^1(\R^n)$ space with the
norm
\begin{equation*}
\|f\|_{L_1^1}=\int_{\mathbb{R}^n}(1+|x|)|f(x)|dx.
\end{equation*}
For simplicity, $||f(\cdot, t)||_{L^p}$ is denoted by
$||f(t)||_{L^p}$ for $1\leq p\leq \infty$, especially, by $||f(t)||$
when $p=2$.

The notion $\widehat{u}(\xi)$ also denotes the Fourier transform
$\CF u(\xi)$ of $u(x)$. For a multi-index
$\al=(\al_1,\cdots,\al_n)$, we denote
\begin{equation*}
 \pa^{\al}=\pa_{x_1}^{\al_1}\cdots\pa_{x_n}^{\al_n},\ \
 \xi^\al=\xi_1^{\al_1}\cdots\xi_n^{\al_n}.
\end{equation*}
The length of $\al$ is $|\al|=\al_1+\cdots+\al_n$. $\be\leq \al$
means $\be_i\leq \al_i$ for all $1\leq i\leq n$. In addition, $\lag
\na \rag$ is defined in terms of the Fourier transform as
\begin{equation*}
\CF \lag \na \rag=\lag \xi\rag:=(1+|\xi|^2)^{\frac{1}{2}}.
\end{equation*}
{}{Finally, for $r\geq 0$, we define $[r]_+$ by
\begin{equation*}
    [r]_+
    =\left\{\begin{array}{ll}
      r, & \ \ \ \text{if $r$ is an integer},\\[3mm]
      {[}r{]}+1,&\ \ \ \text{otherwise},
     \end{array}\right.
\end{equation*}
where $[\cdot]$ means the integer part of the nonnegative argument.}

The rest of the paper is organized as follows. In Section
\ref{sec.l.decay}, we study the time-decay property of solutions to
the linearized equation by using the Fourier analysis. For all $s\in
{\mathbb{R}}$, $L^p$-$L^q$ type time-decay estimates on the linear
solution semigroup are obtained. The result implies that over the
low frequency part, the linear solution semigroup keep the same
algebraic time-decay rate for either $s\leq 1$ or $s>1$, and over
the high frequency part,  it decays with some exponential rate for
$s\leq 1$ but it does so with an algebraic time-rate for  $s>1$
depending on the  regularity degree of initial data.

In Section \ref{sec.non}, the global existence and optimal
time-decay rates of small-amplitude classical solutions to the
nonlinear Cauchy problem \eqref{CP.eq} are established. Depending on
the parameter $s\in \R$, we divide the proof by two cases when
$s\leq 1$ and $s>1$. For the case of $s>1$ which corresponds to the
regularity-loss type, the time-weighted energy method as in
\cite{HoKa,Kuka} is employed to overcome the weakly dissipative
property of the equation. The key part of the proof in this case is
to properly define two integers $N_0$ and $N_1$ such that
time-weighted a priori estimates could be closed.

In Section \ref{s4}, based on the $L^p$-$L^q$ time-decay estimates
on the linear solution semigroup obtained in Section
\ref{sec.l.decay}, we consider the large-time behavior of solutions
asymptotically tending to the linear heat diffusion wave
\eqref{def.heatwave}. {In the last Section \ref{sec.app},} for
completeness, we use an appendix to collect several lemmas with
proofs, which will be frequently used through the paper.


\section{Decay property of linearized solutions}\label{sec.l.decay}

In this section, we study the time-decay property of solutions to the
Cauchy problem on the linearized equation
\begin{equation}\label{LE.eq1}
\left\{\begin{array}{l}
 \dis  \pa_t u-\De P_su=0,\ \  x\in \R^n, t>0,\\[3mm]
\dis u|_{t=0}=u_0,\ \  x\in \R^n.
\end{array}\right.
\end{equation}
Here, initial data $u_0=u_0(x)$ is given. It is easy to see that in terms of the
Fourier transform in $x$, the solution to \eqref{LE.eq1}
is  solved as
\begin{equation*}
\hat{u}(\xi,t)={e}^{-\frac{|\xi|^2}{m(\xi)^s}t}\hat{u}_0(\xi).
\end{equation*}
As usual, the solution semigroup ${e}^{\De P_s t}$ associated with
the linearized Cauchy problem \eqref{LE.eq1} is defined by
\begin{equation*}
{e}^{\De P_s t} u_0=\mathcal{F}^{-1}\left\{{\rm
e}^{-\frac{|\xi|^2}{m(\xi)^s}t}\hat{u}_0(\xi)\right\}.
\end{equation*}

Applying Fourier analysis, we now establish the following
$L^p$-$L^q$ estimate on the solution semigroup ${e}^{\De P_s t}$,
which will play a key role in proving the optimal decay estimate on
the low-order derivatives of the  solutions to the nonlinear Cauchy
problem \eqref{CP.eq}.

\begin{lemma}\label{lem.decay}
Let $n\geq 1$, $k\geq 0$  be integers, $1\leq p,r\leq 2\leq q\leq
\infty$, $\ell\geq 0$ and $s\in \R$. Define $[\cdot]_*$ by
\begin{equation}\label{lem.decay.1}
    \left[\ell+n\left(\frac{1}{r}-\frac{1}{q}\right)\right]_*
    =\left\{\begin{array}{ll}
      \ell, & \ \ \ \text{if $\ell$ is integer and $r=q=2$},\\[3mm]
      {[}\ell+n(\frac{1}{r}-\frac{1}{q}){]}+1,  &\ \ \ \text{otherwise},
     \end{array}\right.
\end{equation}
where $[\cdot]$ means the integer part of the nonnegative argument.
Under the assumption $(\mathcal{M}_1)$ on $m(\xi)$, the solution
semigroup $e^{\De P_s t}$ of the Cauchy problem \eqref{LE.eq1}
satisfies the following time-decay property:
\medskip

\noindent{(i)} When $s\leq 1$,
\begin{equation}\label{lem.decay.2}
\|\na^k e^{\De P_s t} u_0\|_{L^q}\leq C
(1+t)^{-\frac{n}{2}(\frac{1}{p}-\frac{1}{q})-\frac{k}{2}}
\|u_0\|_{L^p}+ C e^{-\la
t}\|\na^{k+[n(\frac{1}{r}-\frac{1}{q})]_*}{u}_0\|_{L^r}
\end{equation}
for any $t\geq 0$, where $C=C(n,k,p,r,q,s)$ and $\la=\la(s)>0$ are constants independent of $u_0$;

\medskip
\noindent{(ii)} When $s>1$,
\begin{multline}\label{lem.decay.3}
\|\na^k e^{\De P_s t} u_0\|_{L^q}\leq C
(1+t)^{-\frac{n}{2}(\frac{1}{p}-\frac{1}{q})-\frac{k}{2}}
\|u_0\|_{L^p}\\
+ C
(1+t)^{-\frac{\ell}{2(s-1)}}\|\na^{k+[\ell+n(\frac{1}{r}-\frac{1}{q})]_*}{u}_0\|_{L^r}
\end{multline}
for any $t\geq 0$, where  $C=C(n,k,p,r,q,\ell,s)$ is a constant independent of $u_0$.
\end{lemma}

\begin{proof}
Take
$2\leq q\leq \infty$ and an integer $k\geq 0$. From Hausdorff-Young
inequality,
\begin{multline}\label{lem.decay.p1}
\|\na^k e^{\De P_s t}u_0\|_{L^q(\R^n_x)}
\leq C\left\||\xi|^ke^{-\frac{|\xi|^2}{m(\xi)^s}t}\hat{u}_0\right\|_{L^{q'}(\R^n_\xi)}\\
\leq C\left\||\xi|^ke^{-\frac{|\xi|^2}{m(\xi)^s}t}\hat{u}_0\right\|_{L^{q'}(|\xi|\leq R)}
+C\left\||\xi|^ke^{-\frac{|\xi|^2}{m(\xi)^s}t}\hat{u}_0\right\|_{L^{q'}(|\xi|\geq R)}\\
:=I_{s,\leq}+I_{s,\geq},
\end{multline}
where $\frac{1}{q}+\frac{1}{q'}=1$ and $R>0$ is arbitrarily chosen. It can be directly observed  that
$\frac{|\xi|^2}{m(\xi)^s}$ satisfies the following estimates:

\medskip
\noindent(i) For any $s\in \R$, there are constants $C_1(s,R)\geq C_0(s,R)>0$ such that
 \begin{equation*}
  C_0(s,R)|\xi|^2 \leq \frac{2|\xi|^2}{m(\xi)^s}\leq C_1(s,R)|\xi|^2
\end{equation*}
holds true over $|\xi|\leq R$. This follows from the fact that $m(\xi)$
is strictly positive and continuous in $\xi\in \R^n$.

\medskip
\noindent(ii) For $s\leq 1$, there is $C_3(s,R)>0$ such that
\begin{equation*}
    \inf_{|\xi|\geq R}\frac{|\xi|^2}{m(\xi)^s}=C_3(s,R).
\end{equation*}
For $s>1$, there is a constant $C_4(s,R)=\frac{1}{m_1(1+1/R)^s}>0$ such that
\begin{equation*}
   \frac{|\xi|^2}{m(\xi)^s}\geq C_4(s,R)|\xi|^{-2(s-1)}
\end{equation*}
holds true over $|\xi|\geq R$. These two lower-bound estimates are
due to the assumption $(\mathcal{M}_1)$ of $m(\xi)$.

\medskip

Now, let us make estimates on $I_{s,\leq}$ and $I_{s,\geq}$. First,
for $I_{s,\leq}$ with $s\in \R$, {by using the standard way, for instance as
in \cite{Ka} or \cite[Lemma 3.1]{Kuka}}, it follows from (i) above that
\begin{multline}\label{lem.decay.p2}
I_{s,\leq}=C\left\||\xi|^ke^{-\frac{|\xi|^2}{m(\xi)^s}t}\hat{u}_0\right\|_{L^{q'}(|\xi|\leq R)}
\leq C\left\||\xi|^ke^{-C_0(s,R)|\xi|^2t}\hat{u}_0\right\|_{L^{q'}(|\xi|\leq R)}\\
\leq C(1+t)^{-\frac{n}{2}(\frac{1}{p}-\frac{1}{q})-\frac{k}{2}}\|u_0\|_{L^p(\R^n_x)},
\end{multline}
where $1\leq p\leq 2$. Next, we use (ii) to estimate $I_{s,\geq}$. When $s\leq 1$, it follows from (ii) above
that
\begin{multline}\label{lem.decay.p3}
  I_{s,\geq}=  C\left\||\xi|^ke^{-\frac{|\xi|^2}{m(\xi)^s}t}\hat{u}_0\right\|_{L^{q'}(|\xi|\geq R)}
  \leq  C e^{-C_3(s,R)t}\left\||\xi|^k \hat{u}_0\right\|_{L^{q'}(|\xi|\geq R)}\\
  \leq C e^{-C_3(s,R)t} \left\||\xi|^{-\frac{r'-q'}{r'q'}(n+\de)}\right\|_{L^{\frac{r'q'}{r'-q'}}(|\xi|\geq R)}
  \left\||\xi|^{k+\frac{r'-q'}{r'q'}(n+\de)}\hat{u}_0\right\|_{L^{r'}(|\xi|\geq R)},
\end{multline}
where H\"{o}lder inequality $1/q'=(r'-q')/r'q'+1/r'$ with $1/r'+1/r=1$ for given $1\leq r\leq 2$, and $\de>0$
is small enough. It is obvious that when $r=q=2$ which implies $r'=q'=2$, it is straightforward to obtain
\begin{equation*}
     I_{s,\geq}\leq C e^{-C_3(s,R)t} \left\||\xi|^k \hat{u}_0\right\|_{L^{r'}(\R^n_\xi)}
     \leq C e^{-C_3(s,R)t}\|\na^k u_0\|_{L^r(\R^n_x)}.
\end{equation*}
Otherwise, when $r\neq2$ or $q\neq 2$ which implies $r<q$ and hence $(r'-q')/r'q'=1/r-1/q>0$,
it follows from \eqref{lem.decay.p3} that
\begin{multline*}
 I_{s,\geq} \leq  C e^{-C_3(s,R)t}
 \left(\int_{|\xi|\geq R}|\xi|^{-(n+\de)}d\xi\right)^{\frac{1}{r}-\frac{1}{q}}
 \left\||\xi|^{k+(\frac{1}{r}-\frac{1}{q})(n+\de)}\hat{u}_0\right\|_{L^{r'}(|\xi|\geq R)}\\
 \leq  C(n,\de, R) e^{-C_3(s,R)t}
 \left\||\xi|^{k+[(\frac{1}{r}-\frac{1}{q})n]+1}\hat{u}_0\right\|_{L^{r'}(\R^n_\xi)}\\
 \leq C(n,\de, R) e^{-C_3(s,R)t}\|\na^{k+[n(\frac{1}{r}-\frac{1}{q})]_*}{u}_0\|_{L^r(\R^n_x)}.
\end{multline*}
Therefore, we finished estimates on $I_{s,\geq}$ for $s\leq 1$,
which together with \eqref{lem.decay.p2}, give \eqref{lem.decay.2}
after plugging them into \eqref{lem.decay.p1}. In order to prove
\eqref{lem.decay.3}, the rest is to estimate $I_{s,\geq}$ with
$s>1$. In fact, when $s>1$, it follows from (ii) above that
\begin{multline}\label{lem.decay.p4}
  I_{s,\geq}=  C\left\||\xi|^ke^{-\frac{|\xi|^2}{m(\xi)^s}t}\hat{u}_0\right\|_{L^{q'}(|\xi|\geq R)}\\
  \leq  C \left\||\xi|^k e^{-C_4(s,R)|\xi|^{-2(s-1)}t} \hat{u}_0\right\|_{L^{q'}(|\xi|\geq R)}\\
   \leq C\sup\limits_{|\xi|\geq
R}\left\{\frac{1}{|\xi|^{\ell}}{\rm
e}^{-C_4(s,R)\frac{t}{|\xi|^{2(s-1)}}}\right\}
\left\||\xi|^{k+\ell} \hat{u}_0\right\|_{L^{q'}(|\xi|\geq R)},
\end{multline}
where $\ell\geq 0$ is fixed. By using the inequality
\begin{equation*}
  \sup\limits_{|\xi|\geq
R}\left\{\frac{1}{|\xi|^{\ell}}{\rm
e}^{-C_4(s,R)\frac{t}{|\xi|^{2(s-1)}}}\right\}\leq C(s,\ell,R) (1+t)^{-\frac{\ell}{2(s-1)}}
\end{equation*}
for any $t\geq 0$, one further has
\begin{equation*}
  I_{s,\geq}\leq C(s,\ell,R) (1+t)^{-\frac{\ell}{2(s-1)}}
  \left\||\xi|^{k+\ell} \hat{u}_0\right\|_{L^{q'}(|\xi|\geq R)}.
\end{equation*}
Similarly to control $I_{s,\geq}$ for $s\leq 1$, it also holds that
\begin{equation*}
    \left\||\xi|^{k+\ell} \hat{u}_0\right\|_{L^{q'}(|\xi|\geq R)}
    \leq C(n,k,\ell, R)\|\na^{k+[\ell+n(\frac{1}{r}-\frac{1}{q})]_*}{u}_0\|_{L^r(\R^n_x)}
\end{equation*}
for $1\leq r\leq 2$. Therefore, one has the estimate on $I_{s,\geq}$
for $s>1$ by
\begin{equation*}
  I_{s,\geq}\leq C(n,k,\ell, R)(1+t)^{-\frac{\ell}{2(s-1)}}
  \|\na^{k+[\ell+n(\frac{1}{r}-\frac{1}{q})]_*}{u}_0\|_{L^r(\R^n_x)},
\end{equation*}
which once again together with \eqref{lem.decay.p2}, prove
\eqref{lem.decay.3} after plugging them into \eqref{lem.decay.p1}.
This completes the proof of Lemma \ref{lem.decay}.
\end{proof}

It is not clear whether there is an explicit representation of
solutions to the Cauchy problem \eqref{LE.eq1} since the explicit
form of the pseudo-differential operator $P_s$ or the frequency
function $m(\xi)$ is unknown. Thus, it is interesting to  find an
asymptotical profile of solutions which has some relatively simple
form. For given $u_0$, to say that $u^{\rm asy}(x, t)$ is a {\it
good} asymptotical profile of the solution $e^{\De P_s t}u_0$, it
means that the rate of $e^{\De P_s t}u_0$ converging to  $u^{\rm
asy}(x, t)$ is strictly larger than that of the solution $e^{\De P_s
t}u_0$ itself decaying to zero in the same space. In our case, an
expected choice is the diffusive wave corresponding to the solution
to some heat equation. In fact, set $\mu_s=\frac{1}{m(0)^s}$. Then,
$u^{\rm asy}(x, t)=e^{\mu_s\De t}u_0$ satisfies
\begin{equation}\label{heat.eq}
\left\{\begin{array}{l}
 \dis  \pa_t u^{\rm asy}-\mu_s\De u^{\rm asy}=0,\ \  x\in \R^n, t>0,\\[3mm]
\dis u^{\rm asy}|_{t=0}=u_0,\ \  x\in \R^n.
\end{array}\right.
\end{equation}
For the time-asymptotical rate of $e^{\De P_s t}u_0$  tending to $e^{\mu_s\De t}u_0$, we have the following

\begin{lemma}
\label{lem.decay.d} Let $n\geq 1$, $k\geq 0$  be integers, $1\leq
p,r\leq 2\leq q\leq \infty$, $\ell\geq 0$ and $s\in \R$. Define $
[\ell+n(\frac{1}{r}-\frac{1}{q})]_*$ as in \eqref{lem.decay.1}.
Under the assumptions $(\mathcal{M}_1)$ and $(\mathcal{M}_2)$ on
$m(\xi)$, the solution semigroup $e^{\De P_s t}$ of the Cauchy
problem \eqref{LE.eq1} asymptotically tends to the heat semigroup
$e^{\mu_s\De t}$ of the Cauchy problem \eqref{heat.eq} with some
time rates stated as follows. Here, in (i) and (ii), $\si>0$ is
given in the assumption $(\mathcal{M}_2)$.

\medskip

\noindent{(i)} When $s\leq 1$,
\begin{multline}\label{lem.decay.d.1}
\|\na^k (e^{\De P_s t}-e^{\mu_s\De t}) u_0\|_{L^q}\leq C
(1+t)^{-\frac{n}{2}(\frac{1}{p}-\frac{1}{q})-\frac{k+\si}{2}}
\|u_0\|_{L^p}\\
+ C e^{-\la t}\|\na^{k+[n(\frac{1}{r}-\frac{1}{q})]_*}{u}_0\|_{L^r}
\end{multline}
for any $t\geq 0$, where $C=C(n,k,p,r,q,s)$ and $\la=\la(s)>0$ are constants independent of $u_0$;

\medskip
\noindent{(ii)} When $s>1$,
\begin{multline}\label{lem.decay.d.2}
\|\na^k (e^{\De P_s t}-e^{\mu_s\De t}) u_0\|_{L^q}\leq C
(1+t)^{-\frac{n}{2}(\frac{1}{p}-\frac{1}{q})-\frac{k+\si}{2}}
\|u_0\|_{L^p}\\
+ C e^{-\la t}\|\na^{k+[n(\frac{1}{r}-\frac{1}{q})]_*}{u}_0\|_{L^r}\\
+ C
(1+t)^{-\frac{\ell}{2(s-1)}}\|\na^{k+[\ell+n(\frac{1}{r}-\frac{1}{q})]_*}{u}_0\|_{L^r}
\end{multline}
for any $t\geq 0$, where  $C=C(n,k,p,r,q,\ell,s)$ is a constant independent of $u_0$.
\end{lemma}

\begin{proof}
Take
an integer $k\geq 0$. Consider
\begin{equation}\label{lem.decay.d.p1}
    \left|\CF\{\na^k (e^{\De P_s t}-e^{\mu_s\De t}) u_0\}\right|=|\xi|^k\left|e^{-\frac{|\xi|^2}{m(\xi)^s}t}
    -e^{-\frac{|\xi|^2}{m(0)^s}t}\right|\cdot |\hat{u}_0|.
\end{equation}
In what follows we estimate the above function over the low frequency part $|\xi|\leq R$ for some $R>0$. One can claim that when $R>0$ is small enough, there are constants $C=C(s,R)>0$, $\la=\la(s,R)>0$ such that
\begin{equation}\label{lem.decay.d.p2}
  \left|e^{-\frac{|\xi|^2}{m(\xi)^s}t}
    -e^{-\frac{|\xi|^2}{m(0)^s}t}\right|\leq C|\xi|^\si e^{-\la |\xi|^2 t}
\end{equation}
holds true over $|\xi|\leq R$. In fact, choose a properly small
$R>0$ such that the assumption $(\mathcal{M}_2)$ holds, that is
\begin{equation*}
    \sup_{|\xi|\leq R}\frac{|m(\xi)-m(0)|}{|\xi|^\si}\leq m_2
\end{equation*}
for constants $\si>0$ and $m_2>0$. Notice
\begin{equation*}
    e^{-\frac{|\xi|^2}{m(\xi)^s}t}
    -e^{-\frac{|\xi|^2}{m(0)^s}t}=-|\xi|^2 t\int_{\frac{1}{m(0)^s}}^{\frac{1}{m(\xi)^s}}
    e^{-\theta |\xi|^2 t}d \theta
\end{equation*}
so that for any $|\xi|\leq R$,
\begin{eqnarray*}
  \left|e^{-\frac{|\xi|^2}{m(\xi)^s}t}
    -e^{-\frac{|\xi|^2}{m(0)^s}t}\right| &\leq & |\xi|^2t \left|\frac{1}{m(\xi)^s}
    -\frac{1}{m(0)^s}\right| e^{-\la_{s,R} |\xi|^2 t}\\
    &\leq &\frac{|s||\xi|^2t}{\min\limits_{|\xi|\leq R}\left[m(\xi)\right]^{s+1}}
    |m(\xi)-m(0)|e^{-\la_{s,R} |\xi|^2 t}\nonumber\\
    &\leq &C(s,R)|\xi|^\si e^{-\frac{\la_{s,R}}{2}|\xi|^2
    t},\nonumber
\end{eqnarray*}
where the strict positivity and continuity of $m(\xi)$ have been used, and $\la_{s,R}$ is defined by
\begin{equation*}
    \la_{s,R}=\min_{|\xi|\leq R}\frac{1}{m(\xi)^s}.
\end{equation*}
Therefore the claim mentioned before follows.

Now, take $2\leq q\leq \infty$ with $1/q+1/q'=1$. Similarly as in
\eqref{lem.decay.p1}, by Hausdorff-Young inequality, it  follows
from \eqref{lem.decay.d.p1} and \eqref{lem.decay.d.p2} that
\begin{multline*}
\|\na^k  (e^{\De P_s t}-e^{\mu_s\De t}) u_0\|_{L^q(\R^n_x)}
\leq C\left\||\xi|^{k+\si}e^{-\la |\xi|^2 t}\hat{u}_0\right\|_{L^{q'}(|\xi|\leq R)}\\
+C\left\||\xi|^ke^{-\mu_s |\xi|^2 t}\hat{u}_0\right\|_{L^{q'}(|\xi|\geq R)}
+C\left\||\xi|^ke^{-\frac{|\xi|^2}{m(\xi)^s}t}\hat{u}_0\right\|_{L^{q'}(|\xi|\geq R)}.
\end{multline*}
Here, three terms on the r.h.s.~can be estimated in the same way as in the proof of Lemma \ref{lem.decay}.
Precisely, the estimate on the first term is similar to \eqref{lem.decay.p2} with $k$ replaced by $k+\si$,
the second term similar to \eqref{lem.decay.p3} and the third term similar to the combination of
\eqref{lem.decay.p3} and \eqref{lem.decay.p4}.
Collecting these estimates imply the desired inequalities \eqref{lem.decay.d.1} and \eqref{lem.decay.d.2}
for $s\leq 1$ and $s>1$,
respectively. For simplicity, all the details are omitted. This completes the proof of Lemma
\ref{lem.decay.d}.
\end{proof}


\section{Nonlinear Cauchy problem}\label{sec.non}

In this section, we are concerned with the global existence and
time-decay rates of small-amplitude classical solutions to the
Cauchy problem \eqref{CP.eq} of the nonlinear scalar conservation
laws for all $n\geq 1$ and all $s\in \R$ under the assumption $(\CM_1)$
on $m(\xi)$. Throughout this section, we always suppose the
assumption $(\CM_1)$ and shall not mention it for simplicity.

\subsection{Uniform-in-time a priori estimates}
In this subsection, by using the energy method as in \cite{DFZ},
we obtain some uniform-in-time a priori estimates on the solution $u(x,t)$
to the Cauchy problem \eqref{CP.eq} for all $n\geq 1$ and $s\in \R$. Recall the equation
\begin{equation}\label{eq.pri}
\pa_t u+\na\cdot f(u)=\De P_s u,\ \ x\in\R^n, t\geq 0.
\end{equation}
In what follows, $u(x,t)$ is supposed to be smooth in $x$, $t$ and satisfy the above equation over
$0\leq t\leq T$ for some $0<T\leq \infty$.

There are two steps in uniform-in-time a priori estimates. The
first step is to estimate the zero-order term, and the second step is to consider the energy estimate of the derivatives on the basis of Lemma \ref{s5.le1} when treating the nonlinear term.

\begin{lemma}\label{lem.pri.0}
There is $\la>0$ such that
\begin{equation}\label{lem.pri.0.1}
    \frac{d}{dt}\|u(t)\|^2+\la \|\na\lag \na \rag^{-s} u(t)\|^2\leq 0
\end{equation}
for any $0\leq t\leq T$.
\end{lemma}

\begin{proof}
As in {\cite[Lemma 2.1]{DFZ}}, the zero-order energy estimate on
\eqref{eq.pri} gives
\begin{equation}\label{lem.pri.0.p1}
    \frac{1}{2}\frac{d}{dt}\|u(t)\|^2+\sum_{j=1}^n\langle f_j(u)_{x_j},
    u\rangle=\langle \De P_s u, u\rangle.
\end{equation}
By Plancherel theorem,
\begin{equation*}
\langle \De P_s u, u\rangle=\langle\widehat{\De P_s
u},\overline{\widehat{u}}
\rangle=-\int_{\R^n}\frac{|\xi|^2}{m(\xi)^s}|\widehat{u}|^2d\xi.
\end{equation*}
Here, using \eqref{m.form},
\begin{multline*}
   -\int_{\R^n}\frac{|\xi|^2}{m(\xi)^s}|\widehat{u}|^2d\xi\leq -\min\left\{\frac{1}{m_0^s},\frac{1}{m_1^s}\right\}
    \int_{\R^n}\frac{|\xi|^2}{(1+|\xi|^2)^s}|\widehat{u}|^2d\xi\\
    =-\min\left\{\frac{1}{m_0^s},\frac{1}{m_1^s}\right\}\|\na\lag \na \rag^{-s} u(t)\|^2.
\end{multline*}
Therefore,
\begin{equation*}
 \langle \De P_s u, u\rangle\leq -\min\left\{\frac{1}{m_0^s},\frac{1}{m_1^s}\right\}\|\na\lag \na \rag^{-s} u(t)\|^2.
\end{equation*}
On the other hand, from integration by part,
\begin{equation*}
 \langle f_j(u)_{x_j},
    u\rangle  =-\int_{\R^n}\left\{\int_0^uf_j(\eta)d\eta\right\}_{x_j}dx=0
\end{equation*}
for each $1\leq j\leq n$. Plugging the above two estimates into \eqref{lem.pri.0.p1}
gives the desired inequality \eqref{lem.pri.0.1}.
\end{proof}

\begin{lemma}\label{lem.pri.k}
Let $k\geq 1$. There are $\la>0$, $C$ such that
\begin{equation}\label{lem.pri.k.1}
    \frac{d}{dt}\sum_{|\al|=k}C^k_\al \|\pa^\al u(t)\|^2+\la
    \|\na^{1+k}\lag \na \rag^{-s}u(t)\|^2\leq C\|\na u(t)\|_{L^\infty}\|\na^k u(t)\|^2
\end{equation}
for any $0\leq t\leq T$, where
$C^k_{\alpha}=\frac{k!}{\alpha!}=\frac{k!}{\alpha_1!\alpha_2!\cdots\alpha_n!}$
for $\alpha=(\alpha_1, \alpha_2, \cdots, \alpha_n)$.
\end{lemma}

\begin{proof}
Take $k\geq 1$. Similarly as in {\cite{DFZ}}, the $k$-order energy
estimate on \eqref{eq.pri} gives
\begin{eqnarray*}
\frac{1}{2}\frac{d}{dt}\sum_{|\al|=k}C_{\al}^{k}\|\pa^\al
u(t)\|^2&+&\sum_{j=1}^n\sum_{|\al|=k}C_{\al}^{k}\langle
\pa_{x_j}\pa^\al f_j(u),\pa^\al u\rangle
=\sum_{|\al|=k}C_{\al}^{k}\langle\De \pa^\al P_su,\pa^\al u\rangle.
\end{eqnarray*}
 Similarly before,
it follows from Plancherel theorem that
\begin{eqnarray*}
\sum_{|\al|=k}C_{\al}^{k}\langle\De \pa^\al P_su,\pa^\al u\rangle
&=&- \sum_{|\al|=k}C_{\al}^{k}\int_{\R^n}
\frac{|\xi|^2}{m(\xi)^s} (i\xi)^\al(\bar{i\xi})^\al
|\widehat{u}|^2d\xi\\
&=&-\int_{\R^n} \frac{|\xi|^{2+2k}}{m(\xi)^s}
|\widehat{u}|^2d\xi,
\end{eqnarray*}
where we used the identity
$(\xi_1^2+\xi_2^2+\cdots+\xi_n^2)^k=\sum\limits_{|\al|=k}C_{\al}^{k}\xi_1^{2\al_1}\xi_2^{2\al_2}\cdots\xi_n^{2\al_n}$.
Then, from \eqref{m.form}, one has
\begin{multline}\label{s3.s.eq3}
\frac{1}{2}\frac{d}{dt}\sum_{|\al|=k}C_{\al}^{k}\|\pa^\al
  u(t)\|^2+\min\left\{\frac{1}{m_0^s},\frac{1}{m_1^s}\right\} \|\na^{1+k}\lag \na \rag^{-s}u(t)\|^2\\
  =\sum_{j=1}^n\sum_{|\al|=k}C_{\al}^{k} I_{j,\al},
\end{multline}
where for $1\leq j\leq n$ and $|\al|=k$, $I_{j,\al}$ is given by
\begin{multline}\label{s3.s.eq4}
I_{j,\al}=\langle \pa_{x_j}\pa^\al f_j(u),-\pa^\al
u\rangle
=\langle f'_j(u)\pa^\al u_{x_j},-\pa^\al u\rangle\\
+\left\langle \sum\limits_{|\beta|=1}^kC_\alpha^\beta\partial^\beta
f'_j(u)\pa^{\al-\beta} u_{x_j},-\pa^\al u\right\rangle
:=I_{j,\al}^1+I_{j,\al}^2.
\end{multline}
It follows from integration by part that
\begin{equation}\label{s3.s.eq5}
I_{j,\al}^1=\frac{1}{2}\int_{\mathbb{R}^n}f''_j(u)u_{x_j}\left(\partial^\alpha
u\right)^2dx\leq C\left\|\nabla
u(t)\right\|_{L^\infty}\left\|\nabla^ku(t)\right\|^2.
\end{equation}
Next, we estimate $ I_{j,\al}^2$ for a fixed $j$ with $1\leq j\leq
n$ and $|\al|=k\geq 1$. In fact, since $|\beta|\geq 1$, the term $\partial^\beta
f'_j(u)\pa^{\al-\beta} u_{x_j}$ can be written as
$$
\partial^{\beta-e_i} \left(f''_j(u)u_{x_i}\right)\pa^{\al-\beta}
u_{x_j}
$$
with $e_i=(0, \cdots, 0, 1, 0, \cdots, 0)$,
$|\beta-e_i|+|\alpha-\beta|=k-1$. Therefore, applying Lemma
\ref{s5.le1} with $p=r=2, q=\infty$ and Corollary \ref{s5.co1}, one has
\begin{eqnarray}\label{s3.s.eq6}
\left\|\partial^{\beta-e_i}
\left(f''_j(u)u_{x_i}\right)\pa^{\al-\beta} u_{x_j}(t)\right\|
&\leq&
C\left\|f''_j(u)u_{x_i}(t)\right\|_{L^\infty}\left\|\nabla^{k-1}u_{x_j}(t)\right\|\nonumber\\
&&+C\left\|u_{x_j}(t)\right\|_{L^\infty}\left\|\nabla^{k-1}\left(f''_j(u)u_{x_i}\right)(t)\right\|\nonumber\\
&\leq&C\left\|\nabla
u(t)\right\|_{L^\infty}\left\|\nabla^ku(t)\right\|.
\end{eqnarray}
Thus, by H\"{o}lder inequality and \eqref{s3.s.eq6},
\begin{eqnarray}\label{s3.s.eq7}
I_{j,\alpha}^2\leq C\left\|\partial^\alpha
u(t)\right\|\left\|\partial^\beta f'_j(u)\pa^{\al-\beta}
u_{x_j}(t)\right\| \leq C\left\|\nabla
u(t)\right\|_{L^\infty}\left\|\nabla^ku(t)\right\|^2.
\end{eqnarray}
Substituting \eqref{s3.s.eq4}, \eqref{s3.s.eq5} and \eqref{s3.s.eq7}
into \eqref{s3.s.eq3} gives the desired estimate \eqref{lem.pri.k.1}.
\end{proof}

\subsection{Existence and time-decay rate for $s\leq 1$}

In this subsection we consider the global existence and time-decay
rates of solutions to the Cauchy problem \eqref{CP.eq} for the case
when $s\leq 1$. The main goal is to prove Theorem \ref{thm.s.small}.

\subsubsection{Global existence}

As the first step, we devote ourselves to the

\medskip

\noindent{\bf Proof of global existence in Theorem
\ref{thm.s.small}:} This follows from the local existence and
uniform-in-time  a priori estimates as well as the continuity
argument. The proof of the local existence is standard, for instance, cf.~\cite{Kato,Kru}, and is thus
omitted for simplicity. It suffices to consider the uniform-in-time
a priori estimates for a smooth solution $u(x,t)$ to the Cauchy
problem \eqref{CP.eq} over $0\leq t\leq T$ for some $0<T\leq
\infty$. From Lemma \ref{lem.pri.0} and Lemma \ref{lem.pri.k}, by
taking the summation of \eqref{lem.pri.0.1} and \eqref{lem.pri.k.1}
with $1\leq k\leq N$, one has
\begin{equation*}
    \frac{d}{dt}\sum_{k=0}^N \sum_{|\al|=k}C_\al^k\|\pa^\al
    u(t)\|^2+\la \sum_{k=0}^N\|\na^{1+k}\lag \na \rag^{-s}u(t)\|^2\leq
    C\|\na u(t)\|_{L^\infty}\|\na u(t)\|_{H^{N-1}}^2.
\end{equation*}
Since
\begin{equation*}
    \sum_{k=0}^N|\na|^{1+k}\lag \na \rag^{-s}=|\na|\lag \na
    \rag^{-s}\sum_{k=0}^N|\na|^k\sim|\na|\lag \na
    \rag^{-s}\lag \na
    \rag^{N}=|\na|\lag\na\rag^{N-s},
\end{equation*}
it further follows that
\begin{equation}\label{thm.s.s.p1}
    \frac{d}{dt}\sum_{k=0}^N \sum_{|\al|=k}C_\al^k\|\pa^\al
    u(t)\|^2+\la\|\na\lag\na\rag^{N-s}u(t)\|^2\leq
    C\|\na u(t)\|_{L^\infty}\|\na u(t)\|_{H^{N-1}}^2
\end{equation}
for any $0\leq t\leq T$. Let us suppose the following smallness a
priori assumption
\begin{equation}\label{thm.s.s.p2}
\sup_{0\leq t\leq T}\|u(t)\|_{H^N}\leq \de
\end{equation}
for a constant $\de>0$ small enough. Due to $N\geq [n/2]+2$,
\begin{equation*}
\sup_{0\leq t\leq T}\|\na u(t)\|_{L^\infty}\leq C\sup_{0\leq t\leq
T}\|u(t)\|_{H^N}\leq C\de.
\end{equation*}
Then, for the case when $s\leq 1$, the right-hand term of
\eqref{thm.s.s.p1} is bounded by
\begin{equation*}
C\de\|\na \lag\na\rag^{N-1}u(t)\|^2\leq C
\de\|\na\lag\na\rag^{N-s}u(t)\|^2,
\end{equation*}
which from \eqref{thm.s.s.p1}, implies
\begin{equation*}
    \frac{d}{dt}\sum_{k=0}^N \sum_{|\al|=k}C_\al^k\|\pa^\al
    u(t)\|^2+\la\|\na\lag\na\rag^{N-s}u(t)\|^2\leq 0
\end{equation*}
since $\de>0$ is small enough. By noticing
\begin{equation*}
\sum_{k=0}^N \sum_{|\al|=k}C_\al^k\|\pa^\al
    u(t)\|^2\sim \|u(t)\|_{H^N}^2,
\end{equation*}
one can take further time integration over $[0,t]$ to obtain
\begin{equation}\label{thm.s.s.p4}
    \|u(t)\|_{H^N}^2+\int_0^t\|\na\lag\na\rag^{N-s}u(\tau)\|^2
    d\tau\leq C\|u_0\|_{H^N}^2
\end{equation}
for any $0\leq t\leq T$. Therefore, in the standard way, as long as
$\|u_0\|_{H^N}$ is sufficiently small, the above uniform-in-time a
priori estimate obtained under the assumption \eqref{thm.s.s.p2}
implies the global existence of solutions by combining the local
existence and uniqueness. This completes the proof of global
existence in Theorem \ref{thm.s.small}.\qed

\subsubsection{Optimal time-decay rates} In order to finish the proof of Theorem
\ref{thm.s.small}, the rest is to prove \eqref{thm.s.small.decay}
for time-decay rates of the obtained solution $u(x,t)$. For that,
define
\begin{equation*}
    \CE_N^{\rm op}(t)=\sup_{0\leq \tau \leq t}\sum_{k=0}^N
    (1+\tau)^{\frac{n}{2}+k}\|\na^ku(\tau)\|^2.
\end{equation*}
One can prove that $ \CE_N^{\rm op}(t)$ is bounded uniformly
in time if $\|u_0\|_{H^N\cap L^1}$ is small enough. In fact, let us
begin with
\begin{equation}\label{thm.s.s.d.p2}
    \frac{d}{dt}\sum_{|\al|=k}C^k_\al \|\pa^\al u(t)\|^2+\la \|\na^{1+k}\lag \na \rag^{-s}
    u(t)\|^2\leq C\chi_{k\geq 1}\|\na u(t)\|_{L^\infty}\|\na^k u(t)\|^2
\end{equation}
for any $t\geq 0$, where $0\leq k\leq N$. This follows from  Lemma
\ref{lem.pri.0} and Lemma \ref{lem.pri.k}. Notice that the
right-hand term of \eqref{thm.s.s.d.p2} vanishes when $k=0$. Let
$a>0$ be a constant to be chosen later. The time-weighted
integration of \eqref{thm.s.s.d.p2} gives
\begin{multline}\label{thm.s.s.d.p3}
(1+t)^{a+k}\|\na^k u(t)\|^2+\int_0^t(1+\tau)^{a+k}\|\na^{1+k}\lag
\na\rag^{-s}u(\tau)\|^2d\tau\\
\leq C\|\na^k
u_0\|^2+C(a+k)\int_0^t(1+\tau)^{a+k-1}\|\na^ku(\tau)\|^2d\tau\\
+C\chi_{k\geq 1} \int_0^t (1+\tau)^{a+k}\|\na
u(\tau)\|_{L^\infty}\|\na^k u(\tau)\|^2d\tau.
\end{multline}
To control the second term on the r.h.s.~of \eqref{thm.s.s.d.p3},
consider
\begin{multline*}
\|\na^k u(t)\|^2=\|\na^{1+k}\lag \na \rag^{-1}u(t)\|^2+\|\na^{k}\lag
\na
\rag^{-1}u(t)\|^2\\
\leq \|\na^{k+1}\lag \na \rag^{-s}u(t)\|^2+\|\na^{k}\lag \na
\rag^{-s}u(t)\|^2
\end{multline*}
since $s\leq 1$. Corresponding to both parts in the sum above, one
need to estimate the following two time integrations:
\begin{eqnarray*}
  I_1&=&C(a+k)\int_0^t (1+\tau)^{a+k-1}\|\na^{k+1}\lag \na
  \rag^{-s}u(\tau)\|^2d\tau,\\
  I_2&=&C(a+k)\int_0^t (1+\tau)^{a+k-1}\|\na^{k}\lag \na
  \rag^{-s}u(\tau)\|^2d\tau.
\end{eqnarray*}
Let $a>1$. $I_1$ is estimated by
\begin{multline*}
    I_1\leq \de \int_0^t (1+\tau)^{a+k}\|\na^{k+1}\lag
    \na\rag^{-s}u(\tau)\|^2d\tau+C_\de \int_0^t\|\na^{k+1}\lag
    \na\rag^{-s}u(\tau)\|^2d\tau\\
\leq \de \int_0^t (1+\tau)^{a+k}\|\na^{k+1}\lag
    \na\rag^{-s}u(\tau)\|^2d\tau+C_\de \|u_0\|_{H^N}^2,
\end{multline*}
where $\de>0$ is an arbitrarily small constant, the Young inequality
with $\frac{a+k-1}{a+k}+\frac{1}{a+k}=1$ was used, and the
inequality
\begin{equation*}
    \sum_{k=0}^N\int_0^t\|\na^{k+1}\lag \na
    \rag^{-s}u(\tau)\|^2d\tau\leq C\|u_0\|_{H^N}^2
\end{equation*}
due to  \eqref{thm.s.s.p4} was also used. For the estimate on $I_2$,
the following interpolation inequality similar to {\cite[Lemma
2.4]{DFZ}} is needed.

\begin{lemma}\label{lem.inter}
For any $k\geq 0$, one has
\begin{equation}\label{lem.inter.1}
\|\na^{k}\lag \na \rag^{-s}u(t)\|\leq C\|\na^{k+1}\lag \na
\rag^{-s}u(t)\|^{\frac{n+2k}{n+2k+2}}\|\hat{u}(t)\|_{L^\infty}^{\frac{2}{n+2k+2}}.
\end{equation}
\end{lemma}

\begin{proof}
Set
\begin{equation*}
    A=\|\na^{k+1}\lag \na \rag^{-s}u(t)\|,\ \
    B=\|\hat{u}(t)\|_{L^\infty}.
\end{equation*}
Then,
\begin{eqnarray*}
 &&\|\na^{k}\lag \na \rag^{-s}u(t)\|_{L^2(\R^n_x)}^2 =
 \int_{\R^n_\xi}\frac{|\xi|^{2k}}{(1+|\xi|^2)^s}|\hat{u}(\xi, t)|^2d\xi\\
&& =\left(\int_{|\xi|\geq R}+\int_{|\xi|\leq R}\right)\frac{|\xi|^{2k}}{(1+|\xi|^2)^s}|\hat{u}(\xi, t)|^2d\xi\nonumber\\
&& \leq \frac{1}{R^2}\int_{|\xi|\geq
 R}\frac{|\xi|^{2k+2}}{(1+|\xi|^2)^s}|\hat{u}(\xi, t)|^2d\xi+\|\hat{u}(t)\|_{L^\infty(\R^n_\xi)}^2\int_{|\xi|\leq R}
 \frac{|\xi|^{2k}}{(1+|\xi|^2)^s}d\xi\nonumber\\
 &&\leq \frac{A^2}{R^2}+CB^2R^{n+2k},\nonumber
\end{eqnarray*}
where $R>0$ can be arbitrary. By taking $R>0$ such that
$\frac{A^2}{R^2}=CB^2R^{n+2k}$, that is
\begin{equation*}
    R=\left(\frac{A}{B}\right)^{\frac{2}{n+2k+2}},
\end{equation*}
it follows that
\begin{equation*}
\|\na^{k}\lag \na \rag^{-s}u(t)\|_{L^2(\R^n_x)}^2\leq
CA^{\frac{2(n+2k)}{n+2k+2}}B^{\frac{4}{n+2k+2}},
\end{equation*}
which is equivalent with \eqref{lem.inter.1}.
\end{proof}

{By the method of \cite{Ito-M3AS,Ito,Kru},} similar to
\cite[Lemma 2.3]{DFZ}, it can be proved that
\begin{equation}\label{bound.L1}
    \|{u}(t)\|_{L^1}\leq \|u_0\|_{L^1}
\end{equation}
for any $t\geq 0$. Now, by plugging \eqref{lem.inter.1} into $I_2$
and using \eqref{bound.L1}, one has
\begin{eqnarray*}
  I_2 &\leq & C \int_0^t(1+\tau)^{a+k-1}\|\na^{k+1}\lag \na
  \rag^{-s}u(\tau)\|^{\frac{2(n+2k)}{n+2k+2}}\|\hat{u}(\tau)\|_{L^\infty}^{\frac{4}{n+2k+2}}d\tau\\
  &\leq &\de \int_0^t(1+\tau)^{a+k}\|\na^{k+1}\lag \na
  \rag^{-s}u(\tau)\|^2d\tau\nonumber\\
  &&+C_\de\int_0^t(1+\tau)^{a+k-\frac{n+2k+2}{2}}\|\hat{u}(\tau)\|_{L^\infty}^2d\tau\nonumber\\
  &\leq &\de \int_0^t(1+\tau)^{a+k}\|\na^{k+1}\lag \na
  \rag^{-s}u(\tau)\|^2d\tau+C_\de
  \|u_0\|_{L^1}^2(1+t)^{a-\frac{n}{2}},\nonumber
\end{eqnarray*}
where  the Young inequality with
$\frac{n+2k}{n+2k+2}+\frac{2}{n+2k+2}=1$ was used and
$a>\frac{n}{2}$ is assumed. Collecting these estimates on $I_1$ and
$I_2$ above, the second term on the r.h.s.~of \eqref{thm.s.s.d.p3}
is bounded by
\begin{multline}\label{thm.s.s.d.p4}
C(a+k)\int_0^t(1+\tau)^{a+k-1}\|\na^ku(\tau)\|^2d\tau \\
\leq 2\de\int_0^t(1+\tau)^{a+k}\|\na^{k+1}\lag \na
  \rag^{-s}u(\tau)\|^2d\tau\\
  +C_\de \|u_0\|_{H^N}^2+C_\de
  \|u_0\|_{L^1}^2(1+t)^{a-\frac{n}{2}}
\end{multline}
for an arbitrarily small constant $\de>0$ and a constant
$a>\max\{1,\frac{n}{2}\}$. Moreover, the third term on the r.h.s.~of
\eqref{thm.s.s.d.p3} is bounded by
\begin{multline}\label{thm.s.s.d.p5}
C\chi_{k\geq 1} \int_0^t (1+\tau)^{a+k}\|\na
u(\tau)\|_{L^\infty}\|\na^k u(\tau)\|^2d\tau\\
\leq C\chi_{k\geq 1} \int_0^t
(1+\tau)^{a+k-\frac{n+1}{2}}\|\na^ku(\tau)\|^2d\tau \sqrt{\CE_N^{\rm
op}(t)}\\
\leq C\chi_{k\geq 1} \int_0^t
(1+\tau)^{a+k-1}\|\na^ku(\tau)\|^2d\tau \sqrt{\CE_N^{\rm op}(t)},
\end{multline}
where $\frac{n+1}{2}\geq 1$ was used and due to $N\geq [n/2]+2$,
\begin{equation*}
    \|\na u(\tau)\|_{L^\infty}\leq C(1+\tau)^{-\frac{n+1}{2}} \sqrt{\CE_N^{\rm op}(t)}
\end{equation*}
for any $0\leq \tau\leq t$. Notice that the time integration term on
the r.h.s.~of \eqref{thm.s.s.d.p5} can again be estimated as in
\eqref{thm.s.s.d.p4}. Thus, under the assumption
\begin{equation}\label{thm.s.s.d.p6}
    \sup_{t\geq 0}\CE_N^{\rm op}(t)\ll 1,
\end{equation}
by plugging \eqref{thm.s.s.d.p4} and \eqref{thm.s.s.d.p5} into
\eqref{thm.s.s.d.p3} and choosing a properly small constant $\de>0$,
one has
\begin{multline*}
(1+t)^{a+k}\|\na^k u(t)\|^2+\int_0^t(1+\tau)^{a+k}\|\na^{1+k}\lag
\na\rag^{-s}u(\tau)\|^2d\tau\\
\leq C\|u_0\|_{H^N}^2+C (1+t)^{a-\frac{n}{2}}\|u_0\|_{L^1}^2
\end{multline*}
for any $t\geq 0$, where $0\leq k\leq N$. Since $a>\frac{n}{2}$, it
follows that
\begin{equation}\label{thm.s.s.d.p7}
    \sup_{t\geq 0}\CE_N^{\rm op}(t)\leq C\|u_0\|_{H^N\cap L^1}^2
\end{equation}
under the assumption \eqref{thm.s.s.d.p6}. Therefore, by the
continuity argument, as long as $\|u_0\|_{H^N\cap L^1}$ is
sufficiently small, \eqref{thm.s.s.d.p7} holds true. Then,
\eqref{thm.s.small.decay} for time-decay rates of the solution
$u(x,t)$ follows by the uniform-in-time boundedness of $\CE_N^{\rm
op}(t)$. This completes the proof of Theorem \ref{thm.s.small}. \qed

\begin{remark}\label{rem.small.L1}
Notice that the right-hand third term of \eqref{thm.s.s.d.p3} can be
neglected  when $k=0$. Then, \eqref{thm.s.s.d.p3} together with
\eqref{thm.s.s.d.p4} imply
\begin{multline*}
(1+t)^{a}\| u(t)\|^2+\int_0^t(1+\tau)^{a}\|\na\lag
\na\rag^{-s}u(\tau)\|^2d\tau\\
\leq C\|
u_0\|^2+C
  \|u_0\|_{L^1}^2(1+t)^{a-\frac{n}{2}},
\end{multline*}
which further yields
\begin{equation*}
 \| u(t)\|\leq C\|u_0\|_{L^2\cap L^1}(1+t)^{-\frac{n}{4}}
\end{equation*}
for any $t\geq 0$. Therefore, for the above time-decay estimate,
the smallness assumption of $\|u_0\|_{L^1}$ can be removed.
This is also consistent with the result in \cite{DFZ} for the case when $s=1$.
\end{remark}


\subsection{Existence and time-decay rate for $s> 1$}

In this subsection, we consider the global existence and time-decay
rates of solutions to the Cauchy problem \eqref{CP.eq} for the case
of $s>1$. The main goal of this subsection is to prove Theorem
\ref{thm.s.large}. Similarly before, this follows from the local existence
and some uniform-in-time a priori estimates with the help of the continuity argument.
Once again, the proof of the local existence is standard and thus omitted for
simplicity.  In what follows, we only consider some  uniform-in-time a priori estimates on
$u(x,t)$ which is supposed to be smooth in $x$, $t$ and satisfy equation
\eqref{eq.pri} over $0\leq t\leq T$ for some $0<T\leq \infty$.

In order to state a priori estimates, let us define three integers $N_0,N$ and $N_1$ in turn
 in terms of $n\geq 1$ and $s>1$ as follows. Notice that $n$
is an integer and $s\in \R$ is a real number and hence might not be an integer. For any integer $k\geq 0$, denote
\begin{eqnarray}
  \ell_1(n,s,k) &=& k+[(\frac{n}{2}+k)(s-1)]_+,\label{def.ll1}\\[2mm]
  \ell_2(n,s,k) &=&k+1+\left\{
  \begin{array}{ll}
    [(\frac{n+3}{2}+k)(s-1)]_+, &\ \ \text{if $n=1$},\\[3mm]
    {[}(\frac{n+1}{2}+k)(s-1)]_+, &\ \ \text{if $n=2$},\\[3mm]
    {[}(\frac{n}{2}+k)(s-1)]_+, &\ \ \text{if $n\geq 3$},
  \end{array}
  \right.\label{def.ll2}\\[2mm]
  \ell_3(s,k) &=& 2+[2(s-1)]+k[s]_+.\label{def.ll3}
\end{eqnarray}
Then, $N_0$ is defined by
\begin{equation}\label{def.N0}
    N_0=\inf\left\{k\in \Z\left|\begin{array}{l}
                       [\frac{k}{[s]_+}]\geq [\frac{n}{2}]+2,\ \
                       k\geq \ell_1(n,s,[\frac{n}{2}]+2),\\[3mm]
                       k\geq \ell_2(n,s,[\frac{n}{2}]+2),\ \text{and}\ \
                       k\geq \ell_3(s,[\frac{n}{2}]+2)
                     \end{array}
    \right.\right\},
\end{equation}
$N$ is arbitrarily chosen such that $N\geq N_0$, and finally $N_1$ is defined by
\begin{equation}\label{def.N1}
    N_1=\sup\left\{k\in \Z\left|\begin{array}{l}
                      [\frac{n}{2}]+2\leq k\leq [\frac{N}{[s]_+}],\ \
                      \ell_1(n,s,k)\leq N,\\[3mm]
                      \ell_2(n,s,k)\leq N,\ \text{and}\ \
                      \ell_3(s,k)\leq N
                     \end{array}
    \right.\right\}.
\end{equation}

Now, let us also define some temporal time-weighted functionals
$\CE_{N}(t)$, $\CD_{N}(t)$, $\CE_{N_1}^{\rm op}(t)$ and $M_i(t)$
$(i=0,1)$ by
\begin{eqnarray}
 \CE_{N}(t)&=& \sum_{k=0}^{[\frac{N}{[s]_+}]}\sup_{0\leq \tau\leq t}(1+\tau)^{k-\frac{1}{2}}\|\na^k u(\tau)\|_{H^{N-k[s]_+}}^2,\label{def.eN}\\
 \CD_N(t)&=&\sum_{k=0}^{[\frac{N}{[s]_+}]+1}\int_0^t (1+\tau)^{k-\frac{3}{2}}\|\na^k \lag\na\rag^{N-k[s]_+}u(\tau)\|d\tau,\nonumber\\
 \CE_{N_1}^{\rm op}(t)&=& \sum_{k=0}^{N_1}\sup_{0\leq \tau\leq t}(1+\tau)^{\frac{n}{2}+k}\|\na^k u(\tau)\|^2,
 \label{def.eN.op}
\end{eqnarray}
and
\begin{equation}\label{def.m01}
    M_i(t)=\sup_{0\leq \tau\leq t} (1+\tau)^{\frac{n+i}{2}}\|\na^i u(\tau)\|_{L^\infty},\ \ i=0,1.
\end{equation}

The key point is to prove that all the above functionals are bounded
uniformly in time if $\|u_0\|_{H^N\cap L^1}$ is small enough. In
fact, the uniform-in-time a priori estimates on these temporal
functionals can be obtained in the following

\begin{lemma}\label{lem.l.pri}
Let $n\geq 1$ and $s>1$. Let $N_0$, $N$ and $N_1$ be defined as before. Then, one has
\begin{equation}\label{lem.l.pri.1}
    \CE_{N}(t)+\CD_N(t)\leq C\|u_0\|_{H^N}^2+CM_1(t)\CD_N(t),
\end{equation}
and
\begin{equation}\label{lem.l.pri.2}
    \CE_{N_1}^{\rm op}(t)\leq C\|u_0\|_{H^N\cap L^1}^2+ C \CE_{N_1}^{\rm op}(t)^2
    +CM_0(t)^2 \CE_{N}(t)
\end{equation}
for any $0\leq t\leq T$.
\end{lemma}

\begin{proof}
We first prove \eqref{lem.l.pri.1}. It is equivalent to prove that for any $0\leq t\leq T$,
\begin{multline}\label{lem.l.pri.p1}
(1+t)^{k-\frac{1}{2}}\left\|\na^k u(t)\right\|_{H^{N-k[s]_+}}^2
+\int_0^t(1+\tau)^{k-\frac{3}{2}}\left\|\nabla^{k}\lag\na\rag^{N-k[s]_+}u(\tau)\right\|^2 d\tau\\
+\int_0^t(1+\tau)^{k-\frac{1}{2}}\left\|\nabla^{k+1}\lag\na\rag^{N-(k+1)[s]_+}u(\tau)\right\|^2 d\tau\\
\leq C\|u_0\|_{H^N}^2+ CM_1(t)\CD_N(t)
\end{multline}
for all $0\leq k\leq [\frac{N}{[s]_+}]$. This can be done by induction on $k$. In fact, similar to obtain \eqref{thm.s.s.p1} from \eqref{eq.pri}
for the case of $s\leq 1$,
it also holds true for the case of $s>1$ that
\begin{equation*}
    \frac{d}{dt}\sum_{k=0}^N \sum_{|\al|=k}C_\al^k\|\pa^\al
    u(t)\|^2+\la\|\na\lag\na\rag^{N-s}u(t)\|^2\leq
    C\|\na u(t)\|_{L^\infty}\|\na u(t)\|_{H^{N-1}}^2.
\end{equation*}
Multiplying the above inequality by
$(1+t)^{-\frac{1}{2}}$ and then taking integration in $t$
gives
\begin{multline*}
(1+t)^{-\frac{1}{2}}\left\|u(t)\right\|_{H^N}^2
+\int_0^t(1+\tau)^{-\frac{3}{2}}\|u(\tau)\|_{H^N}^2d\tau\\
+\int_0^t(1+\tau)^{-\frac{1}{2}}\|\na\lag \na \rag^{N-s} u(\tau)\|^2d\tau\\
\leq C\left\|u_0\right\|_{H^N}^2+
C\int_0^t(1+\tau)^{-\frac{1}{2}}\left\|\nabla
u(\tau)\right\|_{L^\infty}\left\|\nabla
u(\tau)\right\|_{H^{N-1}}^2d\tau,
\end{multline*}
where by using the definitions of $M_1(t), \CD_N(t)$ and the fact that $n\geq 1$,
the last time-integration term is bounded by
\begin{multline*}
C\int_0^t(1+\tau)^{-\frac{1}{2}}\left\|\nabla
u(\tau)\right\|_{L^\infty}\left\|\nabla
u(\tau)\right\|_{H^{N-1}}^2d\tau\\
\leq
CM_1(t)\int_0^t(1+\tau)^{-\frac{n+2}{2}}\left\|
u(\tau)\right\|_{H^N}^2d\tau\\
\leq C\left\|u_0\right\|_{H^N}^2+ CM_1(t)\CD_N(t).
\end{multline*}
Therefore, \eqref{lem.l.pri.p1} with $k=0$ follows  due to $s\leq [s]_+$. Next, suppose that \eqref{lem.l.pri.p1}
is true for $k-1$ with $1\leq k\leq [\frac{N}{[s]_+}]$. From \eqref{thm.s.s.d.p2} which actually also holds for the case of $s>1$, after taking summation from $k$ to $k+N-k[s]_+$, one has
\begin{multline}\label{lem.l.pri.p2}
\frac{d}{dt}\sum_{i=k}^{k+N-k[s]_+}\sum_{|\al|=i}C^i_\al \|\pa^\al
u(t)\|^2
+\la \sum_{i=k}^{k+N-k[s]_+}\|\na^{1+i}\lag \na\rag^{-s}u(t)\|^2\\
\leq C\|\na u(t)\|_{L^\infty}\sum_{i=k}^{k+N-k[s]_+}\|\na^i
u(t)\|^2.
\end{multline}
Notice
\begin{equation*}
   \sum_{i=k}^{k+N-k[s]_+}|\na|^{1+i}\lag \na\rag^{-s}=|\na|^{k+1} \lag \na\rag^{-s}
   \sum_{i=0}^{N-k[s]_+}|\na|^i\sim|\na|^{k+1}\lag \na\rag^{N-k[s]_+-s}
\end{equation*}
which due to $s\leq [s]_+$, implies
\begin{equation*}
 \sum_{i=k}^{k+N-k[s]_+}\|\na^{1+i}\lag \na\rag^{-s}u(t)\|^2
 \geq \la \|\na^{k+1}\lag \na\rag^{N-(k+1)[s]_+}u(t)\|^2.
\end{equation*}
Hence,  multiplying \eqref{lem.l.pri.p2} by $(1+t)^{k-\frac{1}{2}}$ and then taking integration in $t$
yields
\begin{multline}\label{lem.l.pri.p3}
(1+t)^{k-\frac{1}{2}}\left\|\nabla^k
u(t)\right\|_{H^{N-k[s]_+}}^2+\int_0^t(1+\tau)^{k-\frac{1}{2}}
\|\na^{k+1}\lag\na\rag^{N-(k+1)[s]_+}u(\tau)\|^2d\tau\\
\leq C\left\|\nabla^k
u_0\right\|_{H^{N-k[s]_+}}^2+C\int_0^t(1+\tau)^{k-\frac{3}{2}}\left\|\nabla^k
\lag\na\rag^{N-k[s]_+}u(\tau)\right\|^2d\tau\\
+ C\int_0^t(1+\tau)^{k-\frac{1}{2}}\left\|\nabla
u(\tau)\right\|_{L^\infty}\left\|\nabla^k
\lag\na\rag^{N-k[s]_+}u(\tau)\right\|^2d\tau.
\end{multline}
Here, the third term on the r.h.s.~of \eqref{lem.l.pri.p3} is bounded by
\begin{multline*}
 CM_1(t)\int_0^t(1+\tau)^{k-\frac{n+2}{2}}\left\|\nabla^k
\lag\na\rag^{N-k[s]_+}u(\tau)\right\|^2d\tau\\
\leq CM_1(t)\int_0^t(1+\tau)^{k-\frac{3}{2}}\left\|\nabla^k
\lag\na\rag^{N-k[s]_+}u(\tau)\right\|^2d\tau\\
\leq  CM_1(t)\CD_N(t),
\end{multline*}
again from using the definitions of $M_1(t), \CD_N(t)$ and the fact that $n\geq 1$, while for the right-hand
second term of \eqref{lem.l.pri.p3}, by induction assumption for $k-1$, it is bounded by
\begin{equation*}
C\|u_0\|_{H^N}^2+ CM_1(t)\CD_N(t).
\end{equation*}
Therefore, from \eqref{lem.l.pri.p3} as well as the induction
assumption for $k-1$, \eqref{lem.l.pri.p1} is also true for $k$.
Then, by induction on $k$, \eqref{lem.l.pri.p1} holds for all $0\leq
k\leq [\frac{N}{[s]_+}]$. This proves \eqref{lem.l.pri.1}.

Next, to prove \eqref{lem.l.pri.2}, we rewrite \eqref{CP.eq} as a mild form by Duhamel's
principle,
\begin{equation}\label{opt.eq3}
u(x,t)={e}^{\De P_s t}u_0-\sum\limits_{j=1}^n\int_0^t{e}^{\De P_s
(t-\tau)}(g_j(u)_{x_j})(\tau)d\tau.
\end{equation}
Here, $g_j(u)=f_j(u)-f_j(0)-f'_j(0)u$ is set for $1\leq j\leq n$.
Notice that for each $j$, $g_j(u)=O(u^2)$ by the assumption
\eqref{s4.assu.f}. In what follows, fix an integer $k$ with $0\leq
k\leq N_1$. Applying $\nabla^k$ to \eqref{opt.eq3} and taking $L^2$
norm, one has
\begin{multline}\label{opt.eq4}
\left\|\nabla^ku(t)\right\| \leq\displaystyle C\left\|\nabla^k{\rm
e}^{\De P_s
t}u_0\right\|+C\sum_{j=1}^n\int_0^\frac{t}{2}\left\|\nabla^{k+1}{
e}^{\De P_s(t-\tau)}g_j(u)(\tau)\right\|d\tau\\
\displaystyle+C\sum_{j=1}^n\int_\frac{t}{2}^t\left\|\nabla {e}^{\De
P_s(t-\tau)}\nabla^k g_j(u)(\tau)\right\|d\tau :=I_3+I_4+I_5.
\end{multline}
For $I_3$, by applying \eqref{lem.decay.3} with $p=1,r=q=2$ and $\frac{\ell}{2(s-1)}=\frac{n}{4}+\frac{k}{2}$,
it follows that
\begin{equation}\label{opt.eq5}
I_3\leq
C(1+t)^{-\frac{n}{4}-\frac{k}{2}}(\|u_0\|_{L^1}+\|\na^{\ell_1(n,s,k)}u_0\|)\leq
C(1+t)^{-\frac{n}{4}-\frac{k}{2}}\|u_0\|_{H^N\cap L^1},
\end{equation}
where the definitions \eqref{def.ll1} and \eqref{def.N1} for $\ell_1(n,s,k)$ and $N_1$ were used.
For $I_4$, one can apply \eqref{lem.decay.3} with $k$ replaced
by $k+1$ and with  $p=1,r=q=2$ and $\frac{\ell}{2(s-1)}=\frac{n+\ga(n)}{4}+\frac{k}{2}$,
where $\ga(n)=3$ for $n=1$, $1$ for $n=2$ and $0$ for $n\geq 3$, so that it follows
\begin{multline}\label{opt.eq6}
I_4 \leq\displaystyle
C\int_0^\frac{t}{2}(1+t-\tau)^{-\frac{n}{4}-\frac{k+1}{2}}\left\|g(u)(\tau)\right\|_{L^1}d\tau\\
+C\int_0^{\frac{t}{2}}(1+t-\tau)^{-\frac{n+\ga(n)}{4}-\frac{k}{2}}
\|\na^{k+1+[(\frac{n+\ga(n)}{2}+k)(s-1)]_+}g(u)(\tau)\|d\tau\\
=I_{4,1}+I_{4,2}.
\end{multline}
Here and hereafter we used $g(u)$ to denote $g_j(u)$ for $1\leq j\leq n$ without loss of generality.
For the term $I_{4,1}$, it is easy to see
\begin{equation*}
\left\|g(u)(\tau)\right\|_{L^1}\leq C\left\|u(\tau)\right\|^2\leq
C\CE_{N_1}^{\rm op}(t)(1+\tau)^{-\frac{n}{2}},
\end{equation*}
which implies
\begin{multline}\label{opt.eq8}
I_{4,1}\leq\displaystyle
C \CE_{N_1}^{\rm op}(t)\int_0^\frac{t}{2}(1+t-\tau)^{-\frac{n}{4}-\frac{k+1}{2}}(1+\tau)^{-\frac{n}{2}}d\tau\\
\leq\displaystyle C \CE_{N_1}^{\rm op}(t)\times \left\{
\begin{array}{ll}
(1+t)^{-\frac{1}{4}-\frac{k}{2}}, & (n=1)\\[3mm]
(1+t)^{-1-\frac{k}{2}}\ln(1+t), \ \ & (n=2)\\[3mm]
(1+t)^{-\frac{n}{4}-\frac{k+1}{2}}, & (n\geq 3)
\end{array}
\right\} \leq C \CE_{N_1}^{\rm
op}(t)(1+t)^{-\frac{n}{4}-\frac{k}{2}}
\end{multline}
{}for all $n\geq 1$.

On the other hand, for the term $I_{4,2}$, by applying Lemma
\ref{s5.le1}, one has
\begin{eqnarray}\label{opt.eq9}
I_{4,2}&\leq&
C\int_0^{\frac{t}{2}}(1+t-\tau)^{-\frac{n+\ga(n)}{4}-\frac{k}{2}}
\|u(\tau)\|_{L^\infty}\|\na^{\ell_2(n,s,k)}u(\tau)\|d\tau\\
&\leq &CM_0(t)(1+t)^{-\frac{n+\ga(n)}{4}-\frac{k}{2}}\int_0^{\frac{t}{2}}(1+\tau)^{-\frac{n}{2}}\|u(\tau)\|_{H^N}d\tau\nonumber\\
&\leq &CM_0(t)\sqrt{\CE_{N}(t)}(1+t)^{-\frac{n+\ga(n)}{4}-\frac{k}{2}}
\int_0^{\frac{t}{2}}(1+\tau)^{-\frac{n}{2}+\frac{1}{4}}d\tau\nonumber\\
&\leq &
CM_0(t)\sqrt{\CE_{N}(t)}(1+t)^{-\frac{n}{4}-\frac{k}{2}},\nonumber
\end{eqnarray}
where the definitions \eqref{def.m01}, \eqref{def.eN} and \eqref{def.ll2} for $M_0(t)$, $\CE_{N}(t)$
and $\ell_2(n,s,k)$ were used and we also used
\begin{equation}\label{opt.eq9.0}
  \int_0^{\frac{t}{2}}(1+\tau)^{-\frac{n}{2}+\frac{1}{4}}d\tau\leq C (1+t)^{\frac{\ga(n)}{4}}
\end{equation}
for all $n\geq 1$. Then, combining \eqref{opt.eq8}, \eqref{opt.eq9}
with \eqref{opt.eq6} gives
\begin{equation}\label{opt.eq9.1}
    I_4\leq C (1+t)^{-\frac{n}{4}-\frac{k}{2}}\left[\CE_{N_1}^{\rm op}(t)
    + M_0(t)\sqrt{\CE_{N}(t)}\right].
\end{equation}
{}For the term $I_5$, one can apply \eqref{lem.decay.3} with $p=1$,
$r=q=2$ and $\frac{\ell}{2(s-1)}=1+\de$ for a constant $\de>0$ small
enough to be chosen later, so that
\begin{multline}\label{opt.eq10}
I_5 \leq\displaystyle
C\int_\frac{t}{2}^t(1+t-\tau)^{-\frac{n}{4}-\frac{1}{2}}
\left\|\nabla^k g(u)(\tau)\right\|_{L^1}d\tau\\
+C\int_\frac{t}{2}^t(1+t-\tau)^{-(1+\de)}\|\na^{1+k+[2(1+\de)(s-1)]_+}g(u)(\tau)\|d\tau
:=\displaystyle I_{5,1}+I_{5,2}.
\end{multline}
Here, for $I_{5,1}$, since
\begin{equation*}
\left\|\nabla^k g(u)(\tau)\right\|_{L^1} \leq
C\|u(\tau)\|\left\|\nabla^ku(\tau)\right\|
\end{equation*}
by Lemma \ref{s5.le1}, one has
\begin{multline}\label{opt.eq11}
I_{5,1} \leq\displaystyle C
\CE_{N_1}^{\rm op}(t)\int_\frac{t}{2}^t(1+t-\tau)^{-\frac{n}{4}-\frac{1}{2}}
(1+\tau)^{-\frac{n}{2}-\frac{k}{2}}d\tau\\
\leq\displaystyle C\CE_{N_1}^{\rm op}(t)\times\left\{
\begin{array}{ll}
(1+t)^{-\frac{1}{4}-\frac{k}{2}}, & (n=1)\\[3mm]
(1+t)^{-1-\frac{k}{2}}\ln(1+t), \ \ & (n=2)\\[3mm]
(1+t)^{-\frac{n}{2}-\frac{k}{2}}, & (n\geq 3)
\end{array}
\right\}
\leq C\CE_{N_1}^{\rm op}(t)(1+t)^{-\frac{n}{4}-\frac{k}{2}}
\end{multline}
{}for all $n\geq 1$.

To estimate $I_{5,2}$, notice that from the definitions
\eqref{def.ll3} and \eqref{def.N1} of $\ell_3(s,k)$ and $N_1$, one
can take $\de>0$ small enough such that
\begin{equation*}
    1+[2(1+\de)(s-1)]_++k[s]_+\leq\ell_3(s,k)\leq N
\end{equation*}
for any $0\leq k\leq N_1$. Then, it follows from Lemma \ref{s5.le1}
that
\begin{multline*}
\|\na^{1+k+[2(1+\de)(s-1)]_+}g(u)(\tau)\|\leq
C\|u(\tau)\|_{L^\infty}\|\na^{1+k+[2(1+\de)(s-1)]_+}u(\tau)\|\\
\leq CM_0(t)(1+\tau)^{-\frac{n}{2}}\|\na^k
u(\tau)\|_{H^{N-k[s]_+}}\\
\leq
CM_0(t)\sqrt{\CE_{N}(t)}(1+\tau)^{-\frac{n}{2}-\frac{k}{2}+\frac{1}{4}}
\end{multline*}
for any $0\leq \tau\leq t$. This further implies
\begin{multline*}
I_{5,2}\leq
CM_0(t)\sqrt{\CE_{N}(t)}\int_{\frac{t}{2}}^t(1+t-\tau)^{-(1+\de)}
(1+\tau)^{-\frac{n}{2}-\frac{k}{2}+\frac{1}{4}}d\tau\\
\leq
CM_0(t)\sqrt{\CE_{N}(t)}(1+t)^{-\frac{n}{4}-\frac{k}{2}-\frac{n-1}{4}}
\int_0^{\frac{t}{2}}(1+\tau)^{-(1+\de)}d\tau\\
\leq CM_0(t)\sqrt{\CE_{N}(t)}(1+t)^{-\frac{n}{4}-\frac{k}{2}},
\end{multline*}
where $n\geq 1$ was used. By plugging estimates on $I_{5,1}$ and
$I_{5,2}$ into \eqref{opt.eq10}, one has
\begin{equation}\label{opt.eq11.1}
    I_5\leq
    C\left[\CE_{N_1}^{\rm op}(t)+M_0(t)\sqrt{\CE_{N}(t)}\right](1+t)^{-\frac{n}{4}-\frac{k}{2}}.
\end{equation}
Thus, \eqref{opt.eq4} together with \eqref{opt.eq5},
\eqref{opt.eq9.1} and \eqref{opt.eq11.1} yields that for any $0\leq
t\leq T$,
\begin{equation*}
(1+t)^{\frac{n}{4}+\frac{k}{2}}\|\na^ku(t)\|\leq C\|u_0\|_{H^N\cap
L^1}+C\CE_{N_1}^{\rm op}(t)+CM_0(t)\sqrt{\CE_{N}(t)}
\end{equation*}
with $0\leq k\leq N_1$. Since the right-hand term of the above
estimate is nondecreasing in $t$, it further holds that
\begin{equation*}
\sqrt{\CE_{N_1}^{\rm op}(t)}\leq C\|u_0\|_{H^N\cap
L^1}+C\CE_{N_1}^{\rm op}(t)+CM_0(t)\sqrt{\CE_{N}(t)}
\end{equation*}
for any $0\leq t\leq T$. This proves \eqref{lem.l.pri.2} and hence
completes the proof of Lemma \ref{lem.l.pri}.
\end{proof}

\noindent{\bf Proof of Theorem \ref{thm.s.large}:} As mentioned at
the beginning of this subsection, it suffices to consider the
uniform-in-time a priori estimates on the smooth solution $u(x,t)$
to the Cauchy problem \eqref{CP.eq} over $0\leq t\leq T$ for
$0<T\leq \infty$. Due to the definition \eqref{def.N1} of $N_1$,
$N_1\geq [\frac{n}{2}]+2$ holds true. By using the Sobolev
inequality as in {\cite[Proposition 3.8]{Ta}}, it follows
that for any $0\leq t\leq T$,
\begin{equation*}
    M_0(t)+M_1(t)\leq C \sqrt{\CE_{N_1}^{\rm op}(t)}.
\end{equation*}
Then, from Lemma \ref{lem.l.pri}, \eqref{lem.l.pri.1} and \eqref{lem.l.pri.2} imply
\begin{equation*}
    \CE_{N}(t)+\CD_N(t)\leq C\|u_0\|_{H^N}^2+C\sqrt{\CE_{N_1}^{\rm op}(t)}\CD_N(t),
\end{equation*}
and
\begin{equation*}
    \CE_{N_1}^{\rm op}(t)\leq C\|u_0\|_{H^N\cap L^1}^2+ C \CE_{N_1}^{\rm op}(t)^2
    +C\CE_{N_1}^{\rm op}(t) \CE_{N}(t),
\end{equation*}
respectively. By setting
\begin{equation*}
    X(t)= \CE_{N}(t)+\CD_N(t)+ \CE_{N_1}^{\rm op}(t),
\end{equation*}
the above two inequalities further lead to
\begin{equation*}
    X(t)\leq C\left(\|u_0\|_{H^N\cap L^1}^2+X(t)^{\frac{3}{2}}+X(t)^2\right)
\end{equation*}
for any $0\leq t\leq T$. From the continuity argument, it is easy to
see that $X(t)$ is bounded uniformly in time under the assumption
that $\|u_0\|_{H^N\cap L^1}$ is small enough. Therefore, the global
existence of solutions as in \eqref{thm.s.large.1} follows  by the
standard way and also the optimal time-decay estimate
\eqref{thm.s.large.2} results from the definition \eqref{def.eN.op}
of $\CE_{N_1}^{\rm op}(t)$. The proof of Theorem \ref{thm.s.large}
is complete.\qed

\medskip

Up to now, we have obtained the existence and optimal decay
rates of the global solutions to the Cauchy problem \eqref{CP.eq}
for all spatial dimensions $n\geq 1$ and for all $s\in\mathbb{R}$.

\section{Large-time asymptotic behavior}\label{s4}

In this section  we shall prove Theorem \ref{thm.asymptotic} on the
large-time behavior of the obtained solutions. For that purpose, we
divide the proof by several steps in order to time-asymptotically approximate the solution to the Cauchy
problem \eqref{CP.eq}.

First of all, we prove that the solution to the nonlinear Cauchy
problem \eqref{CP.eq} can be approximated by the one to the
corresponding linearized problem at infinite time. For given
$u_0=u_0(x)$, let us define $\tilde{u}={e}^{\De P_st}u_0$ to be the
solution to the linearized Cauchy problem corresponding to
\eqref{CP.eq} by
\begin{equation*}
\left\{\begin{array}{l}
 \dis  \pa_t \tilde{u}-\De P_s\tilde{u}=0,\ \  x\in \R^n, t>0,\\[3mm]
\dis \tilde{u}|_{t=0}=u_0,\ \  x\in \R^n.
\end{array}\right.
\end{equation*}
Then, one has the following two lemmas which correspond to the case
when  $s\leq 1$ and $s>1$, respectively.

\begin{lemma}[case for $s\leq1$]\label{asy.small.lem}
Let $n\geq 2$, $s\leq 1$, and $N\geq \left[\frac{n}{2}\right]+2$.
Suppose that $\|u_0\|_{H^N\cap L^1}$ is sufficiently small, and
$u(x,t)$ is a solution to the Cauchy problem \eqref{CP.eq} obtained
in Theorem \ref{thm.s.small}. Then, for any $t\geq 0$,
\begin{equation}\label{asy.small.dec}
\left\|\nabla^k\left(u-{e}^{\De P_st}u_0\right)(t)\right\|\leq
C\rho(t)(1+t)^{-\frac{n}{4}-\frac{k+1}{2}},
\end{equation}
where $0\leq k\leq N$, and $\rho(t)=\ln(1+t)$ for $n=2$ and
$\rho(t)=1$ for $n\geq 3$.
\end{lemma}

\begin{proof}
Take $0\leq k\leq N$. Similar to obtain \eqref{opt.eq4}, it follows
from \eqref{opt.eq3} that
\begin{eqnarray}\label{asy.sm.prf.eq1}
\left\|\nabla^k\left(u(t)-e^{\De P_s t}u_0\right)\right\|
&\leq&\displaystyle I_4'+I_5',
\end{eqnarray}
where $I_4'$, $I_5'$ have the same definitions as $I_4$, $I_5$ in
\eqref{opt.eq4}. In what follows we would improve the previous
estimates on  $I_4$, $I_5$ with the help of the obtained time-decay
estimate \eqref{thm.s.small.decay}. For $I_4'$, one can apply
\eqref{lem.decay.2} with $k$ replaced by $k+1$ and with $p=1,r=q=2$,
so that
\begin{multline}\label{asy.sm.prf.eq2}
I_4' \leq\displaystyle
C\int_0^\frac{t}{2}(1+t-\tau)^{-\frac{n}{4}-\frac{k+1}{2}}\left\|g(u)(\tau)\right\|_{L^1}d\tau\\
+C\int_0^{\frac{t}{2}}e^{-\lambda
(t-\tau)}\|\na^{k+1}g(u)(\tau)\|d\tau :=I_{4,1}'+I_{4,2}',
\end{multline}
where for the term $I_{4,1}'$, from Lemma \ref{s5.le1} and
\eqref{thm.s.small.decay},
\begin{multline}\label{asy.sm.prf.eq3}
I_{4,1}' \leq
C\int_0^\frac{t}{2}(1+t-\tau)^{-\frac{n}{4}-\frac{k+1}{2}}\left\|u(\tau)\right\|^2d\tau \\
\leq
C\int_0^\frac{t}{2}(1+t-\tau)^{-\frac{n}{4}-\frac{k+1}{2}}(1+\tau)^{-\frac{n}{2}}d\tau
\leq C\rho(t)(1+t)^{-\frac{n}{4}-\frac{k+1}{2}},
\end{multline}
and for the term $I_{4,2}'$, by applying Lemma \ref{s5.le1} and
\eqref{thm.s.small.decay}, one has
\begin{equation}\label{asy.sm.prf.eq4}
I_{4,2}' \leq\displaystyle C\int_0^{\frac{t}{2}}e^{-\lambda
(t-\tau)}\|u(\tau)\|_{L^\infty}\|\na^{k+1}u(\tau)\|d\tau\\
\leq C(1+t)^{-\frac{n}{4}-\frac{k+1}{2}}.
\end{equation}
Then, plugging \eqref{asy.sm.prf.eq3} and \eqref{asy.sm.prf.eq4}
into \eqref{asy.sm.prf.eq2} gives
\begin{equation}\label{asy.sm.prf.eq5}
I_4' \leq C\rho(t)(1+t)^{-\frac{n}{4}-\frac{k+1}{2}}.
\end{equation}
For the term $I_5'$, one can apply \eqref{lem.decay.2} with $p=1$,
$r=q=2$, so that
\begin{multline}\label{asy.sm.prf.eq6}
I_5' \leq\displaystyle
C\int_\frac{t}{2}^t(1+t-\tau)^{-\frac{n}{4}-\frac{1}{2}}
\left\|\nabla^k g(u)(\tau)\right\|_{L^1}d\tau\\
+C\int_{\frac{t}{2}}^te^{-\lambda
(t-\tau)}\|\na^{k+1}g(u)(\tau)\|d\tau :=I_{5,1}'+I_{5,2}'.
\end{multline}
Here, for $I_{5,1}'$, since
\begin{equation*}
\left\|\nabla^k g(u)(\tau)\right\|_{L^1} \leq
C\|u(\tau)\|\left\|\nabla^ku(\tau)\right\|
\end{equation*}
by Lemma \ref{s5.le1}, one has
\begin{multline*}
I_{5,1}' \leq\displaystyle C
\int_\frac{t}{2}^t(1+t-\tau)^{-\frac{n}{4}-\frac{1}{2}}
(1+\tau)^{-\frac{n}{2}-\frac{k}{2}}d\tau\\
\leq\displaystyle C\times\left\{
\begin{array}{ll}
(1+t)^{-1-\frac{k}{2}}\ln(1+t), \ \ & n=2,\\[3mm]
(1+t)^{-\frac{n}{2}-\frac{k}{2}}, & n\geq 3.
\end{array}
\right.
\end{multline*}
{}For $I_{5,2}'$, similar to \eqref{asy.sm.prf.eq4}, one has
\begin{equation*}
I_{5,2}' \leq C\int_{\frac{t}{2}}^te^{-\lambda
(t-\tau)}\|u(\tau)\|_{L^\infty}\|\na^{k+1}u(\tau)\|d\tau \leq
C(1+t)^{-\frac{n}{4}-\frac{k+1}{2}}.
\end{equation*}
By putting estimates on $I_{5,1}'$ and $I_{5,2}'$ into
\eqref{asy.sm.prf.eq6}, one has
\begin{equation}\label{asy.sm.prf.eq9}
I_5' \leq C\rho(t)(1+t)^{-\frac{n}{4}-\frac{k+1}{2}}.
\end{equation}
Thus, \eqref{asy.sm.prf.eq1} together with \eqref{asy.sm.prf.eq5}
and \eqref{asy.sm.prf.eq9} yields \eqref{asy.small.dec} and hence
completes the proof of Lemma \ref{asy.small.lem}.
\end{proof}

Now, we turn to the case when $s>1$. To the end, similar to define
$N_1$ in \eqref{def.N1}, let us define the integer $N_2$  in terms
of $n\geq 2$, $s>1$ and $N$ as follows. Denote
\begin{eqnarray}
  &&\dis \ell_4(n,s,k) =k+1+{[}(\frac{n+1}{2}+k+1)(s-1)]_+,\label{def.ll4}\\
& &\dis \ell_5(s,k) = \ell_3(s,k)+\nu(s,k), \ \ \nu(s,k):=\left\{
\begin{array}{ll}
(k+1)\left([s]_+-1\right), \ \ &\  n=2,\\[2mm]
k\left([s]_+-1\right), &\  n= 3,4,\\[2mm]
0, &\  n\geq 5.
\end{array}
\right.\label{def.ll5}
\end{eqnarray}
Then, $N_2$ is defined by
\begin{equation}\label{def.N2}
    N_2=\sup\left\{k\in \Z\left|\begin{array}{l}
                      0\leq k\leq [\frac{N}{[s]_+}],\ \
                      \ell_4(n,s,k)\leq N,\\[3mm]
                      \text{and}\ \
                      \ell_5(s,k)\leq N
                     \end{array}
    \right.\right\}.
\end{equation}

\begin{lemma}[case for $s>1$]\label{asy.large.lem}
Let $n\geq 2$, $s>1$  and $N\geq N_0$, where $N_0$ is given in
\eqref{def.N0}. Suppose that $\|u_0\|_{H^N\cap L^1}$ is sufficiently
small, and  $u(x,t)$ is a solution to the Cauchy problem
\eqref{CP.eq} obtained in Theorem \ref{thm.s.large}. Then, for any
$t\geq 0$,
\begin{equation}\label{asy.large.dec}
\left\|\nabla^k\left(u-e^{\De P_s t}u_0\right)(t)\right\|\leq
C\rho(t)(1+t)^{-\frac{n}{4}-\frac{k+1}{2}},
\end{equation}
where $0\leq k\leq N_2$, and $\rho(t)$ is defined in the same way as
in Lemma \ref{asy.small.lem}.
\end{lemma}

\begin{proof}
Take $1\leq k\leq N_2$. Similar to obtain \eqref{asy.sm.prf.eq1},
one has
\begin{eqnarray}\label{asy.lg.prf.eq1}
\left\|\nabla^k\left(u(t)-e^{\De P_s t}u_0\right)\right\|
&\leq&\displaystyle I_4^{''}+I_5^{''},
\end{eqnarray}
where as for $I_4^{'}$ and $I_5^{'}$ in  \eqref{asy.sm.prf.eq1},
$I_4^{''}$ and $I_5^{''}$ also have the same definitions as $I_4$
and $I_5$ given in \eqref{opt.eq4}. Once again, in what follows we
would improve the previous estimates on  $I_4$, $I_5$ with the help
of the obtained time-decay estimate \eqref{thm.s.small.decay} in the
case when $s>1$. For $I_4^{''}$ that is $I_4$ in \eqref{opt.eq4},
one can apply \eqref{lem.decay.3} with $k$ replaced by $k+1$ and
with  $p=1,r=q=2$ and
$\frac{\ell}{2(s-1)}=\frac{n+\ga(n)}{4}+\frac{k+1}{2}$ for $n\geq
2$, so that it follows
\begin{multline}\label{asy.lg.prf.eq2}
I_4^{''} \leq\displaystyle
C\int_0^\frac{t}{2}(1+t-\tau)^{-\frac{n}{4}-\frac{k+1}{2}}\left\|g(u)(\tau)\right\|_{L^1}d\tau\\
+C\int_0^{\frac{t}{2}}(1+t-\tau)^{-\frac{n+\ga(n)}{4}-\frac{k+1}{2}}
\|\na^{k+1+[(\frac{n+\ga(n)}{2}+k+1)(s-1)]_+}g(u)(\tau)\|d\tau\\
:=I_{4,1}^{''}+I_{4,2}^{''}.
\end{multline}
For the term $I_{4,1}^{''}$,  as in \eqref{opt.eq8}, one has
\begin{multline}\label{asy.lg.prf.eq22}
I_{4,1}^{''} \leq\displaystyle C \CE_{N_1}^{\rm op}(t)\times \left\{
\begin{array}{ll}
(1+t)^{-1-\frac{k}{2}}\ln(1+t), \ \ & n=2,\\[3mm]
(1+t)^{-\frac{n}{4}-\frac{k+1}{2}}, & n\geq 3
\end{array}
\right.\\
\leq C \CE_{N_1}^{\rm
op}(t)\rho(t)(1+t)^{-\frac{n}{4}-\frac{k+1}{2}},
\end{multline}
where the last inequality holds true since $n\geq 2$. For the term
$I_{4,2}^{''}$, by applying Lemma \ref{s5.le1}, one has
\begin{eqnarray}\label{asy.lg.prf.eq3}
I_{4,2}^{''}&\leq&
C\int_0^{\frac{t}{2}}(1+t-\tau)^{-\frac{n+\ga(n)}{4}-\frac{k+1}{2}}
\|u(\tau)\|_{L^\infty}\|\na^{\ell_4(n,s,k)}u(\tau)\|d\tau\\
&\leq &CM_0(t)(1+t)^{-\frac{n+\ga(n)}{4}-\frac{k+1}{2}}
\int_0^{\frac{t}{2}}(1+\tau)^{-\frac{n}{2}}\|u(\tau)\|_{H^N}d\tau\nonumber\\
&\leq
&CM_0(t)\sqrt{\CE_{N}(t)}(1+t)^{-\frac{n+\ga(n)}{4}-\frac{k+1}{2}}
\int_0^{\frac{t}{2}}(1+\tau)^{-\frac{n}{2}+\frac{1}{4}}d\tau\nonumber\\
&\leq &
CM_0(t)\sqrt{\CE_{N}(t)}(1+t)^{-\frac{n}{4}-\frac{k+1}{2}},\nonumber
\end{eqnarray}
where the definitions \eqref{def.m01}, \eqref{def.eN} and
\eqref{def.ll4} for $M_0(t)$, $\CE_{N}(t)$ and $\ell_4(n,s,k)$ were
used and we also used \eqref{opt.eq9.0} for all $n\geq 2$. Then,
combining \eqref{asy.lg.prf.eq2}, \eqref{asy.lg.prf.eq22} with
\eqref{asy.lg.prf.eq3} gives
\begin{equation}\label{asy.lg.prf.eq31}
    I_4^{''}\leq C \rho(t)(1+t)^{-\frac{n}{4}-\frac{k+1}{2}}\left[\CE_{N_1}^{\rm op}(t)
    + M_0(t)\sqrt{\CE_{N}(t)}\right].
\end{equation}

Next, we turn to estimate $I_5^{''}$ in \eqref{asy.lg.prf.eq1} or
equivalently $I_5$ in \eqref{opt.eq4}. As in \eqref{opt.eq10}, we
write for simplicity that $I_5^{''}$ is bounded by the sum of
$I_{5,1}^{''}$ and $I_{5,2}^{''}$ which have the same definitions as
 $I_{5,1}$ and $I_{5,2}$, respectively. Then, similar to obtain \eqref{opt.eq11}, it
follows that
\begin{multline}\label{asy.lg.prf.eq32}
I_{5,1}^{''} \leq\displaystyle C\CE_{N_1}^{\rm op}(t)\times\left\{
\begin{array}{ll}
(1+t)^{-1-\frac{k}{2}}\ln(1+t), \ \ & n=2,\\[3mm]
(1+t)^{-\frac{n}{2}-\frac{k}{2}}, & n\geq 3
\end{array}
\right.\\
\leq C\CE_{N_1}^{\rm
op}(t)\rho(t)(1+t)^{-\frac{n}{4}-\frac{k+1}{2}},
\end{multline}
where $n\geq 2$ was used for the last inequality. To estimate
$I_{5,2}^{''}$, notice that from the definitions \eqref{def.ll5} and
\eqref{def.N2} of $\ell_5(s,k)$ and $N_2$, one can take $\de>0$
small enough such that
\begin{equation*}
    1+[2(1+\de)(s-1)]_++k[s]_+\leq\ell_5(s,k)\leq N
\end{equation*}
for any $0\leq k\leq N_2$. Then, it follows from Lemma \ref{s5.le1}
that
\begin{multline*}
\|\na^{1+k+[2(1+\de)(s-1)]_+}g(u)(\tau)\|\leq
C\|u(\tau)\|_{L^\infty}\|\na^{1+k+[2(1+\de)(s-1)]_+}u(\tau)\|\\
\leq CM_0(t)(1+\tau)^{-\frac{n}{2}}\times \left\{
\begin{array}{l}
\|\na^{k+1} u(\tau)\|_{H^{N-(k+1)[s]_+}}, \ \ n=2,\\[2mm]
\|\na^k u(\tau)\|_{H^{N-k[s]_+}}, \ \ n=3,4,\\[2mm]
\|\na^{k-1} u(\tau)\|_{H^{N-(k-1)[s]_+}}, \ \ n\geq 5
\end{array}
\right.\\
\leq
CM_0(t)\sqrt{\CE_{N}(t)}(1+\tau)^{-\frac{n}{2}-\frac{\eta(k)}{2}+\frac{1}{4}}
\end{multline*}
for any $0\leq \tau\leq t$, where $\eta(k)=k+1$ for $n=2$, $k$ for
$n=3,4$ and $k-1$ for $n\geq 5$. This further implies
\begin{multline*}
I_{5,2}^{''}\leq
CM_0(t)\sqrt{\CE_{N}(t)}\int_{\frac{t}{2}}^t(1+t-\tau)^{-(1+\de)}
(1+\tau)^{-\frac{n}{2}-\frac{\eta(k)}{2}+\frac{1}{4}}d\tau\\
\leq
CM_0(t)\sqrt{\CE_{N}(t)}(1+t)^{-\frac{n}{4}-\frac{k+1}{2}-\frac{n+2\eta(k)-2(k+1)-1}{4}}
\int_0^{\frac{t}{2}}(1+\tau)^{-(1+\de)}d\tau\\
\leq CM_0(t)\sqrt{\CE_{N}(t)}(1+t)^{-\frac{n}{4}-\frac{k+1}{2}},
\end{multline*}
where $n\geq 2$ was used. Therefore, by combining estimates on
$I_{5,1}^{''}$ and $I_{5,2}^{''}$ above, one has
\begin{equation}\label{asy.lg.prf.eq5}
    I_5^{''}\leq I_{5,1}^{''}+I_{5,2}^{''}\leq C \rho(t)(1+t)^{-\frac{n}{4}-\frac{k+1}{2}}\left[\CE_{N_1}^{\rm op}(t)
    + M_0(t)\sqrt{\CE_{N}(t)}\right].
\end{equation}
Thus, \eqref{asy.lg.prf.eq1} together with {\eqref{asy.lg.prf.eq31}}
and \eqref{asy.lg.prf.eq5} yield that for any $0\leq t\leq T$,
\begin{eqnarray*}
\left\|\nabla^k\left(u(t)-e^{\De P_s t}u_0\right)\right\|&\leq&
    C\left[\CE_{N_1}^{\rm op}(t)
    + M_0(t)\sqrt{\CE_{N}(t)}\right]\rho(t)(1+t)^{-\frac{n}{4}-\frac{k+1}{2}}\\
    &\leq&
C\rho(t)(1+t)^{-\frac{n}{4}-\frac{k+1}{2}}\nonumber
\end{eqnarray*}
with $0\leq k\leq N_2$, where the uniform-in-time boundedness of
{$\CE_{N_1}^{\rm op}(t), M_0(t)$} and $\CE_{N}(t)$  was used. This
proves \eqref{asy.large.dec} and hence completes the proof of Lemma
\ref{asy.large.lem}.
\end{proof}

Next, for the initial data $u_0$ given above, we define the desired
time-asymptotic profile $u^*=u^*(x,t)$ by
\begin{equation*}
u^*(x,t)=G(x,t+1)\int_{\mathbb{R}^n}u_0(x)dx,
\end{equation*}
where $G=G(x,t)=(4\mu_s\pi
t)^{-\frac{n}{2}}e^{-\frac{|x|^2}{4\mu_st}}$ is the usual Green
function of the linear heat equation. As used in
\cite{Liu-Kawashima}, we have the following well-known result.

\begin{lemma}\label{asy.heat.lem}
Let $n\geq 1$, $k\geq 0$ and $1\leq q\leq 2$. If $\phi\in
L^q\left(\mathbb{R}^n\right)$, then
\begin{equation*}
\left\|\na_x^kG*\phi (t)\right\|\leq
Ct^{-\frac{n}{2}\left(\frac{1}{q}-\frac{1}{2}-\frac{k}{2}\right)}\|\phi\|_{L^q}
\end{equation*}
for any $t>0$. Also, if $\phi\in L_1^1\left(\mathbb{R}^n\right)$ and
$\int_{\mathbb{R}^n}\phi(x)dx=0$, then
\begin{equation*}
\left\|\na_x^kG*\phi(t)\right\|\leq
Ct^{-\frac{n}{2}\left(\frac{1}{q}-\frac{1}{2}-\frac{k+1}{2}\right)}\|\phi\|_{L_1^1},
\end{equation*}
for any $t>0$.
\end{lemma}

Based on Lemma \ref{asy.heat.lem}, one can show that $e^{\mu_s\De
t}u_0$ is well approximated by $u^*(x,t)$ time-asymptotically.

\begin{lemma}\label{asy.profile.lem}
Let $n\geq 1$, $k\geq 0$, and let $u_0\in
L_1^1\left(\mathbb{R}^n\right)$ and $\int_{\mathbb{R}^n}u_0(x)dx=0$.
Then,
\begin{equation}\label{asy.profile.dec}
\left\|\nabla^k\left(e^{\mu_s\De t}u_0-u^*\right)(t)\right\|\leq
C(1+t)^{-\frac{n}{4}-\frac{k+1}{2}}
\end{equation}
for any $t>0$.
\end{lemma}

\begin{proof}
Define
\begin{equation*}
\phi_0(x)=G(x,1)=(4\mu_s\pi)^{-\frac{n}{2}}e^{-\frac{|x|^2}{4\mu_s}}.
\end{equation*}
Then, $G(x,t+1)=G(t)*\phi_0(x)$ holds true. Consequently, one can
write
\begin{equation*}
e^{\mu_s\De t}u_0-u^*=G(t)*(u_0-M\phi_0),\ \
M:=\int_{\mathbb{R}^n}u_0(x)dx.
\end{equation*}
Note that $\int_{\mathbb{R}^n}\phi_0(x)dx=1$ and
$\int_{\mathbb{R}^n}(u_0-M\phi_0)dx=0$. Therefore, for $k\geq 0$, by
applying Lemma \ref{asy.heat.lem}, we deduce that
\begin{equation*}
\left\|\nabla^k\left(e^{\mu_s\De
t}u_0-u^*(t)\right)\right\|=\left\|\nabla^kG*(u_0-M\phi_0)(t)\right\|\leq
C(1+t)^{-\frac{n}{4}-\frac{k+1}{2}}
\end{equation*}
for any $t\geq 0$. This proves \eqref{asy.profile.dec} and hence
completes the proof of Lemma \ref{asy.profile.lem}.
\end{proof}

\noindent{\bf Proof of Theorem \ref{thm.asymptotic}:} For the
solution $u$ to the Cauchy problem \eqref{CP.eq} and the desired
time-asymptotic profile $u^\ast$, their difference can be rewritten
as
\begin{equation}\label{asy.profile.prf.1}
u-u^*=\left(u-e^{\De P_st}u_0\right)+\left(e^{\De
P_st}u_0-e^{\mu_s\De t}u_0\right)+\left(e^{\mu_s\De
t}u_0-u^*\right).
\end{equation}
Suppose that $\|u_0\|_{H^N\cap L^1}$ is small enough and $u_0\in
L^1_1$ with $\int_{\R^n}u_0\,dx=0$. When $s\leq 1$, combining of
\eqref{asy.small.dec}, \eqref{lem.decay.d.1},
\eqref{asy.profile.dec} with \eqref{asy.profile.prf.1} yields
\eqref{thm,asy.large.dec} for $0\leq k\leq N-1$. When $s> 1$,
combining of \eqref{asy.large.dec}, \eqref{lem.decay.d.2},
\eqref{asy.profile.dec} with \eqref{asy.profile.prf.1} leads to
\eqref{thm,asy.large.dec} for $0\leq k\leq N_2$. The proof of
Theorem \ref{thm.asymptotic} is complete.

\section{Appendix}\label{sec.app}
In this section,  we prove some inequalities about $L^p$ upper
bounds of some nonlinear terms, which have been used in the previous
sections. This first inequality is about the $L^p$ estimate on any
two product terms with the sum of the order of their derivatives
equal to a given integer.

\begin{lemma}\label{s5.le1}
Let $n\geq 1$. Let
$\alpha^1=\left(\alpha_1^1,\cdots,\alpha_n^1\right)$ and
$\alpha^2=\left(\alpha_1^2,\cdots,\alpha_n^2\right)$ be two
multi-indices with $\left|\alpha^1\right|=k_1,
\left|\alpha^2\right|=k_2$ and set $k=k_1+k_2$. Let $1\leq
p,q,r\leq\infty$ with $1/p=1/q+1/r$. Then, for $u_j: \R^n\to \R$
$(j=1,2)$, one has
\begin{multline}\label{s5.eq1}
\left\|\partial^{\alpha^1}u_1\partial^{\alpha^2}u_2\right\|_{L^p\left(\mathbb{R}^n\right)}\\
\leq C\left(\left\|u_1\right\|_{L^q\left(\mathbb{R}^n\right)}
\left\|\nabla^ku_2\right\|_{L^r\left(\mathbb{R}^n\right)}
+\left\|u_2\right\|_{L^q\left(\mathbb{R}^n\right)}
\left\|\nabla^ku_1\right\|_{L^r\left(\mathbb{R}^n\right)}\right)
\end{multline}
for a constant $C$ independent of $u_1$ and $u_2$.
\end{lemma}
\begin{proof}
It is similar to the proof for the case of  $n=1$ given in
\cite[Lemma 4.1]{HoKa}. Here, for the convenience of readers, we
present a complete proof {for the multi-dimensional version}. First,
for $j=1,2$, set $\theta_j=\frac{k_j}{k}$ and define $p_j$ by
\begin{equation*}
\frac{1}{p_j}=\frac{1-\theta_j}{q}+\frac{\theta_j}{r}.
\end{equation*}
It is straightforward to verify
$\frac{1}{p}=\frac{1}{p_1}+\frac{1}{p_2}$ since
$\theta_1+\theta_2=1$. Therefore, applying the H\"{o}lder
inequality, the Gagliardo-Nirenberg inequality and the Young
inequality, one has
\begin{eqnarray*}
&&\left\|\partial^{\alpha_1}u_1\partial^{\alpha_2}u_2\right\|_{L^p\left(\mathbb{R}^n\right)}
\leq\left\|\partial^{\alpha^1}u_1\right\|_{L^{p_1}\left(\mathbb{R}^n\right)}
\left\|\partial^{\alpha^2}u_2\right\|_{L^{p_2}\left(\mathbb{R}^n\right)}\\
&&\ \ \ \ \ \  \leq
C\left(\left\|u_1\right\|_{L^q\left(\mathbb{R}^n\right)}^{1-\theta_1}
\left\|\nabla^ku_1\right\|_{L^r\left(\mathbb{R}^n\right)}^{\theta_1}\right)
\left(\left\|u_2\right\|_{L^q\left(\mathbb{R}^n\right)}^{1-\theta_2}
\left\|\nabla^ku_2\right\|_{L^r\left(\mathbb{R}^n\right)}^{\theta_2}\right)\\
&&\ \ \ \ \ \
=C\left(\left\|u_1\right\|_{L^q\left(\mathbb{R}^n\right)}
\left\|\nabla^ku_2\right\|_{L^r\left(\mathbb{R}^n\right)}\right)^{\theta_2}
\left(\left\|u_2\right\|_{L^q\left(\mathbb{R}^n\right)}
\left\|\nabla^ku_1\right\|_{L^r\left(\mathbb{R}^n\right)}\right)^{\theta_1}\\
&&\ \ \ \ \ \  \leq
C\left(\left\|u_1\right\|_{L^q\left(\mathbb{R}^n\right)}
\left\|\nabla^ku_2\right\|_{L^r\left(\mathbb{R}^n\right)}+\left\|u_2\right\|_{L^q\left(\mathbb{R}^n\right)}
\left\|\nabla^ku_1\right\|_{L^r\left(\mathbb{R}^n\right)}\right),
\end{eqnarray*}
where the following Gagliardo-Nirenberg inequality over
$\mathbb{R}^n$ was used:
\begin{equation*}
\left\|\nabla^{k_j}u_j\right\|_{L^{p_j}\left(\mathbb{R}^n\right)}
\leq\left\|\nabla^ku_j\right\|_{L^r\left(\mathbb{R}^n\right)}^{\theta_j}
\left\|u_j\right\|_{L^q\left(\mathbb{R}^n\right)}^{1-\theta_j}
\end{equation*}
with $\frac{1}{p_j}=\frac{k_j}{n}
+\theta_j\left(\frac{1}{r}-\frac{k}{n}\right)+(1-\theta_j)\frac{1}{q}$,
$j=1,2$. Then \eqref{s5.eq1} follows.  The proof of Lemma
\ref{s5.le2} is complete.
\end{proof}

The second inequality is the generalization of Lemma \ref{s5.le1} up
to the case of products of several terms.

\begin{lemma}\label{s5.le2}
Let $n\geq 1$, $l\geq 2$ be integers. Let
$\alpha^j=(\alpha_1^j,\cdots,\alpha_n^j)$, $1\leq j\leq l$, be
multi-indices with $|\alpha^j|=k_j$,$1\leq j\leq l$, and set
$k=k_1+\cdots+k_l$. Let $1\leq p,q,r\leq\infty$ with $1/p=1/q+1/r$.
Then, for $\mathbf{u}=(u_1,\cdots,u_l): \R^n\to \R^l$, one has
\begin{equation}\label{s1.eq2}
\left\|\prod\limits_{j=1}^l\partial^{\alpha^j}u_j\right\|_{L^p\left(\mathbb{R}^n\right)}
\leq
C\left\|\mathbf{u}\right\|_{L^\infty\left(\mathbb{R}^n\right)}^{l-2}
\left\|\mathbf{u}\right\|_{L^q\left(\mathbb{R}^n\right)}
\left\|\nabla^k\mathbf{u}\right\|_{L^r\left(\mathbb{R}^n\right)}
\end{equation}
for a constant $C$ independent of $\mathbf{u}$.
\end{lemma}
\begin{proof}
Once again, it is similar to the proof for the case of  $n=1$ as in
\cite[Lemma 4.1]{IgKa}. In fact, for $j=1,\cdots,l$, set
$\theta_j=\frac{k_j}{k}$ and define $p_j$ by
\begin{equation*}
\frac{1}{p_j}=\frac{1-\theta_j}{(l-1)q}+\frac{\theta_j}{r}.
\end{equation*}
Notice $\frac{1}{p}=\frac{1}{p_1}+\cdots+\frac{1}{p_l}$ since
$\theta_1+\cdots+\theta_l=1$. Therefore, applying the H\"{o}lder
inequality and the Gagliardo-Nirenberg inequality, one has
\begin{eqnarray*}
\left\|\prod\limits_{j=1}^l\partial^{\alpha^j}u_j\right\|_{L^p\left(\mathbb{R}^n\right)}
&\leq&\prod\limits_{j=1}^l\left\|\partial^{\alpha^j}u_j\right\|_{L^{p_j}\left(\mathbb{R}^n\right)}\\
&\leq&C\prod\limits_{j=1}^l\left(\left\|u_j\right\|_{L^{(l-1)q}\left(\mathbb{R}^n\right)}^{1-\theta_j}
\left\|\nabla^ku_j\right\|_{L^r\left(\mathbb{R}^n\right)}^{\theta_j}\right)\\
&\leq&C\prod\limits_{j=1}^l\left(\left\|u_j\right\|_{L^\infty\left(\mathbb{R}^n\right)}^{\frac{l-2}{l-1}(1-\theta_j)}
\left\|u_j\right\|_{L^q\left(\mathbb{R}^n\right)}^{\frac{1-\theta_j}{l-1}}
\left\|\nabla^ku_j\right\|_{L^r\left(\mathbb{R}^n\right)}^{\theta_j}\right)\\
&\leq&C\prod\limits_{j=1}^l\left(\left\|\mathbf{u}\right\|_{L^\infty
\left(\mathbb{R}^n\right)}^{\frac{l-2}{l-1}(1-\theta_j)}
\left\|\mathbf{u}\right\|_{L^q\left(\mathbb{R}^n\right)}^{\frac{1-\theta_j}{l-1}}
\left\|\nabla^k\mathbf{u}\right\|_{L^r\left(\mathbb{R}^n\right)}^{\theta_j}\right)\\
&=&
C\left\|\mathbf{u}\right\|_{L^\infty\left(\mathbb{R}^n\right)}^{l-2}
\left\|\mathbf{u}\right\|_{L^q\left(\mathbb{R}^n\right)}
\left\|\nabla^k\mathbf{u}\right\|_{L^r\left(\mathbb{R}^n\right)},
\end{eqnarray*}
where similarly before, we used the Gagliardo-Nirenberg inequality
over $\mathbb{R}^n$
\begin{equation*}
\left\|\nabla^{k_j}u_j\right\|_{L^{p_j}\left(\mathbb{R}^n\right)}
\leq\left\|\nabla^ku_j\right\|_{L^r\left(\mathbb{R}^n\right)}^{\theta_j}
\left\|u_j\right\|_{L^q\left(\mathbb{R}^n\right)}^{1-\theta_j}
\end{equation*}
with
$$
\frac{1}{p_j}=\frac{k_j}{n}+\theta_j
\left(\frac{1}{r}-\frac{k}{n}\right)+(1-\theta_j)\frac{1}{(l-1)q}
$$
for each $j=1,\cdots,l$. This proves Lemma \ref{s5.le2}.
\end{proof}

{}From Lemma \ref{s5.le2}, we have the following corollary, which
provides a limit situation of \eqref{s1.eq2}.

\begin{corollary}\label{s5.co1}
Let $n\geq 1$ and $1\leq p\leq\infty$. Let
$\alpha=\left(\alpha_1,\cdots,\alpha_n\right)$ be a multi-index with
$\left|\alpha\right|=k$. Assume that $F(u)$ is a smooth function of
$u$. Then, there is a constant
$C(\|u\|_{L^\infty\left(\mathbb{R}^n\right)})$ depending only on
$\|u\|_{L^\infty\left(\mathbb{R}^n\right)}$ with
$C(\varepsilon)\rightarrow 0$ when $\varepsilon\rightarrow 0$ such
that
\begin{eqnarray*}
\left\|\pa^{\alpha}\left(F(u)u_{x_i}\right)\right\|_{L^p\left(\mathbb{R}^n\right)}
\leq C(\|u\|_{L^\infty\left(\mathbb{R}^n\right)})
\left\|\nabla^{k+1}u\right\|_{L^p\left(\mathbb{R}^n\right)}
\end{eqnarray*}
for all $1\leq i\leq n$.
\end{corollary}
\begin{proof}
By Leibniz formula, for each  $1\leq i\leq n$,
\begin{multline*}
\pa^{\alpha}\left(F(u)u_{x_i}\right) =F(u)\partial^\alpha
u_{x_i}+\sum\limits_{|\beta|=1}^k
C_\alpha^\beta\partial^\beta F(u)\partial^{\alpha-\beta}u_{x_i}\\
=F(u)\partial^\alpha
u_{x_i}+\sum\limits_{|\beta|=1}^kC_\alpha^\beta\sum_{l=1}^{|\be|}F^{(l)}(u)
\sum_{\substack{\alpha^1+\cdots+\alpha^l=\be
\\|\alpha_j|\geq 1, 1\leq j\leq
l}}C_{\alpha^1,\cdots,\alpha^l}\prod_{j=1}^{l}\pa^{\alpha^j}
u\partial^{\alpha-\beta}u_{x_i}.
\end{multline*}
Define
$\alpha^{l+1}=\alpha-\beta+e_i=\alpha-\beta+(0,\cdots,0,1,0,\cdots,0)$.
Then, it follows from Lemma \ref{s5.le2} that
\begin{eqnarray*}
\left\|\pa^{\alpha}\left(F(u)u_{x_i}\right)\right\|_{L^p\left(\mathbb{R}^n\right)}
&\leq& \left\|F(u)\partial^\alpha
u_{x_i}\right\|_{L^p\left(\mathbb{R}^n\right)}
+C\sum\limits_{|\beta|=1}^k\sum_{l=1}^{|\be|}\left\|\prod\limits_{j=1}^{l+1}
\partial^{\alpha^j}u_j\right\|_{L^p\left(\mathbb{R}^n\right)}\\
&\leq& C(\|u\|_{L^\infty\left(\mathbb{R}^n\right)})
\left\|\nabla^{k+1}u\right\|_{L^p\left(\mathbb{R}^n\right)}.
\end{eqnarray*}
This proves Corollary \ref{s5.co1}.
\end{proof}

\section*{Acknowledgements}
\ \ \ \ The first author was supported by the start-up fund from
CUHK. The second author was supported by the Natural Science
Foundation of China (The Youth Foundation) $\#$10901068 and CCNU
Project (No. CCNU09A01004) and the third author was supported by the
National Natural Science Foundation of China $\#$10625105,
$\#$11071093, the PhD specialized grant of the Ministry of Education
of China $\#$20100144110001, and the self-determined research funds
of CCNU from the colleges'basic research and operation of MOE.


\end{document}